\documentclass[a4paper,11pt]{amsart}

\usepackage[latin1]{inputenc}
\usepackage[english]{babel}
\usepackage{tikz}
\usepackage[all]{xy}

\newenvironment{nouppercase}{%
  \renewcommand{\uppercasenonmath}[1]{\LARGE}}{}

\usepackage[T1]{fontenc}
\usepackage{amsmath}
\usepackage{amsfonts}
\usepackage{amssymb}
\usepackage{amsthm,eepic}
\usepackage{mathrsfs}
\usepackage{enumitem}
\usepackage{graphicx}
\usepackage{footnote}
\usepackage{hyperref}
\hypersetup{colorlinks=true,    linkcolor=blue,
citecolor=red,       filecolor=BrickRed,       urlcolor=darkgreen
}
\usepackage[normalem]{ulem}
\usepackage{cancel}

\renewcommand{\geq}{\geqslant}
\renewcommand{\leq}{\leqslant}
\newtheorem{thm}{Theorem}

\newtheorem{prop}[thm]{Proposition}
\newtheorem{lem}[thm]{Lemma}
\newtheorem{assu}[thm]{Assumption}

\def\K{\mathbb{K}}
\def\Z{\mathbb{Z}}
\def\N{\mathbb{N}}
\def\C{\mathbb{C}}
\def\R{\mathbb{R}}
\def\Q{\mathbb{Q}}
\def\P1{\mathbb{P}^{1}}
\def\Etproj{\overline{E}_{t}}
\def\iup{{\widetilde{\iota}}}
\def\rx{r_{x}}
\def\ry{r_{y}}

\newcommand{\cx}{{\rm\textbf{\textsf{x}}}}



\makeatletter
\def\testb#1{\testb@i#1,,\@nil}%
\def\testb@i#1,#2,#3\@nil{%
  \draw[->, thick] (O) --++(#1);
  \ifx\relax#2\relax\else\testb@i#2,#3\@nil\fi}
\makeatother   
 
\title[Enumeration of weighted quadrant walks]{Enumeration of weighted quadrant walks:\\criteria for algebraicity and D-finiteness}

\date{\today}

\author{Thomas Dreyfus}
\address{Universit\'e de Bourgogne, CNRS, Institut de Math\'ematiques de Bourgogne, Dijon, France}
\email{thomas.dreyfus@cnrs.fr}

\author{Andrew Elvey Price}
\address{Universit\'e de Tours, CNRS, Institut Denis Poisson, Tours, France}
\email{andrew.elvey-price@cnrs.fr}

\author{Kilian Raschel}
\address{CNRS, International Research Laboratory France-Vietnam in mathematics and its applications, Vietnam Institute for Advanced Study in Mathematics, Hano\"i, Vietnam}
\email{raschel@math.cnrs.fr}
\thanks{This work has received funding from the European Research Council (ERC)\ under the European Union's Horizon 2020 research and innovation programme under the Grant Agreement No.~759702, from Centre Henri Lebesgue, programme ANR-11-LABX-0020-0 and from the project DeRerumNatura, programme ANR-19-CE40-0018. The IMB has received support from the EIPHI Graduate School (contract ANR-17-EURE-0002).}
\begin{document}
\begin{nouppercase}
\maketitle
\end{nouppercase}

\begin{abstract}
In the field of enumeration of weighted walks confined to the quarter plane, it is known that the generating functions behave very differently depending on the chosen step set; in practice, the techniques used in the literature depend on the complexity of the counting series. In this paper we introduce a unified approach based on the theory of elliptic functions, which allows us to have a common proof of the characterisation of the algebraicity and D-finiteness of the generating functions.
\end{abstract}

\section{Introduction and main results}
\label{sec:intro}

\subsection*{Walks in the positive quadrant}
In this paper we consider weighted walks in the quarter plane and their associated generating functions, and provide necessary and sufficient conditions for the latter series to be D-finite (i.e., solution of a linear differential equation with polynomial coefficients)\ or even algebraic (solution of a polynomial equation).
A small step walk (or path)\ in the quarter plane $\N^2=\{0,1,2,\ldots\}^2$ is a sequence of points $P_0,P_1,\ldots,P_n$, each $P_k$ lying in the quarter plane, the steps $P_{k+1}-P_k$ belonging to a given finite step set $\mathcal S\subset \{0, \pm1\}^2$. See Figure~\ref{fig:traj} for an illustration. Such objects are very natural in both combinatorics and probability theory: they are interesting in themselves, and also because they are strongly related to other discrete structures \cite{BMM,DeWa-15}.

\begin{figure}[!t]
\centering
\includegraphics[width=0.5\textwidth]{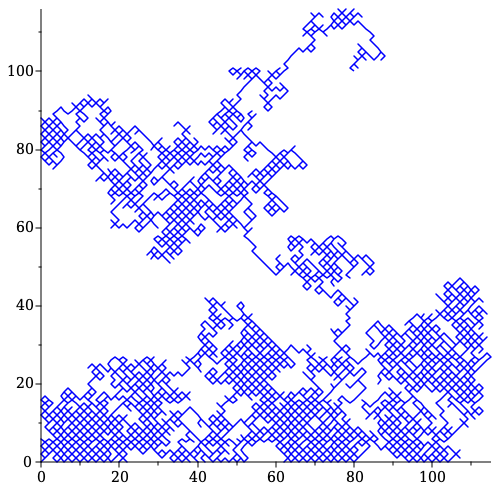}
\caption{A walk of length $n=10,000$ started at $P_0=(0,0)$ with jumps in $\mathcal S=\{\pm1\}^2$ and uniform transition probabilities $d_{i,j}=\frac{1}{4}$}
\label{fig:traj}
\end{figure}

To each step $(i,j)\in\{0,\pm 1\}^{2}$ we assign a weight $d_{i,j}\geq0$, which can be seen as the probability that the walk goes in the direction $(i,j)$. With a renormalisation we can assume that $\sum d_{i,j}=1$.  The model is called unweighted if all non-zero $d_{i,j}$ have the same values. The weight of a given (finite length)\ path is defined as the product of the weights of its component steps. For any $(i,j)\in \N^{2}$ and any $n\geq 0$, we let  
\begin{equation}
\label{eq:def_prob}
   \mathbb P\bigl(P_0 \xrightarrow{n}(i,j)\bigr)
\end{equation}
be the probability that the walk started at $P_0$ (often taken as the origin $(0,0)$)\ is at some generic position $(i,j)$ after the $n$-th step, with all intermediate points $P_k$ remaining in the cone. In other words, the probability in \eqref{eq:def_prob} is the sum of the weights of all paths reaching the position $(i,j)$ from the initial position $(0,0)$ after $n$ steps.
 More specifically, we will turn our attention to the generating (or counting)\ function
\begin{equation}
\label{eq:def_generating_function}
   Q(x,y,t) = \sum_{i,j,n\geq 0} \mathbb P\bigl(P_0 \xrightarrow{n}(i,j)\bigr)x^iy^jt^n.
\end{equation}

\subsection*{Classification of quadrant walk models}
There is a large literature on (mostly unweighted)\ walks in the quarter plane, focusing on various probabilistic and combinatorial aspects. To summarise, two main questions have attracted the attention of the mathematical community: first, finding an exact  expression for the probability~\eqref{eq:def_prob}, or equivalently for the series~\eqref{eq:def_generating_function}. The second question is to characterise the algebraic nature of the series~\eqref{eq:def_generating_function}, according to the classes of functions 
\begin{equation*}
    \{\text{rational}\} \subset \{\text{algebraic}\} \subset \{\text{D-finite}\} \subset \{\text{D-algebraic}\}.
\end{equation*} 
The first question, which is combinatorial in nature, should not overshadow the second. Understanding the nature of $Q(x, y, t)$ has implications for the asymptotic behaviour of the coefficients (see, for example, reference \cite{BRS14}), and also allows us to gain insight into the complexity of these lattice path problems (to illustrate this fact, let us recall that unconstrained walks are associated with rational generating functions, while walks confined to a half-plane admit algebraic counting functions \cite{BMP}). This is the second question we will consider in the present work. To be precise, the function $Q(x,y,t)$ is said to be D-finite (resp.\ D-algebraic, or differentially algebraic)\ if it satisfies a linear (resp.\ algebraic)\ differential equation with polynomial coefficients in $\mathbb C(x,y,t)$, in each of its three variables $x,y,t$. A function is said to be differentially transcendental if it is not D-algebraic.

Throughout this paper, in some results we will assume that a certain group of transformations (simply related to the weights)\ is finite or infinite; this group will be properly introduced in Section~\ref{secnot}. This group was introduced in \cite{FIM17} in a probabilistic context, and used in a crucial way for combinatorial purposes in the seminar paper \cite{BMM}. Intuitively, finite group models are those to which a generalisation of the well-known reflection principle applies. Finite group models of order $4$, $6$ and $8$ are fully characterised in \cite{zbMATH06577211}; see Figure~\ref{fig:step_sets} for some examples. Moreover, it is proved in \cite[Rem.~5.1]{hardouin2021differentially} that the order of the group cannot exceed $12$ (see also \cite{JiTaZh-21}). What is still missing is a complete description of the parameters that lead to a group of order $10$ and $12$. There is, however, a characterisation of the finiteness of the group in terms of the existence of a certain invariant, see \cite[Lem.~2.4]{hardouin2021differentially}.

\begin{figure}
\begin{center}
 \begin{tikzpicture}[scale=.7] 
    \draw[->,white] (1,2) -- (0,-2);
    \draw[->,white] (1,-2) -- (0,2);
    \draw[->] (0,0) -- (1,1);
    \draw[->] (0,0) -- (1,0);
    \draw[->] (0,0) -- (-1,-1);
    \draw[->] (0,0) -- (-1,0);
    \draw[->] (0,0) -- (0,1);
    \draw[->] (0,0) -- (-1,1);
    \draw[->] (0,0) -- (1,-1);
    \node at (-1.2,0) {$3$};
    \node at (1.2,1.3) {$15$};
    \node at (-1.2,-1.4) {$2$};
    \node at (0,1.3) {$13$};
    \node at (1.2,0) {$9$};
    \node at (1.2,-1.4) {$6$};
    \node at (-1.2,1.3) {$5$};
   \end{tikzpicture}\qquad
\begin{tikzpicture}[scale=.7] 
    \draw[->,white] (1,2) -- (0,-2);
    \draw[->,white] (1,-2) -- (0,2);
    \draw[->] (0,0) -- (0,1);
    \draw[->] (0,0) -- (1,1);
    \draw[->] (0,0) -- (1,0);
    \draw[->] (0,0) -- (1,-1);
    \draw[->] (0,0) -- (0,-1);
    \draw[->] (0,0) -- (-1,-1);
    \node at (-1.2,-1.4) {$1$};
    \node at (1.2,1.3) {$7$};
    \node at (0,-1.4) {$2$};
    \node at (0,1.3) {$7$};
    \node at (1.2,0) {$5$};
    \node at (1.2,-1.4) {$1$};
   \end{tikzpicture}
   \qquad
\begin{tikzpicture}[scale=.7] 
    \draw[->,white] (1,2) -- (0,-2);
    \draw[->,white] (1,-2) -- (0,2);
    \draw[->] (0,0) -- (1,1);
    \draw[->] (0,0) -- (1,0);
    \draw[->] (0,0) -- (-1,0);
    \draw[->] (0,0) -- (-1,-1);
    \node at (-1.2,0) {$6$};
    \node at (1.2,0) {$2$};
    \node at (1.2,1.3) {$4$};
    \node at (-1.2,-1.4) {$3$};
   \end{tikzpicture}
   \qquad
\begin{tikzpicture}[scale=.7] 
    \draw[->,white] (1,2) -- (0,-2);
    \draw[->,white] (1,-2) -- (0,2);
    \draw[->] (0,0) -- (1,1);
    \draw[->] (0,0) -- (1,0);
    \draw[->] (0,0) -- (0,-1);
    \draw[->] (0,0) -- (-1,0);
    \draw[->] (0,0) -- (0,1);
    \draw[->] (0,0) -- (-1,1);
    \draw[->] (0,0) -- (1,-1);
    \node at (-1.2,0) {$1$};
    \node at (1.2,1.3) {$1$};
    \node at (0,-1.4) {$1$};
    \node at (0,1.3) {$2$};
    \node at (1.2,0) {$2$};
    \node at (1.2,-1.4) {$1$};
    \node at (-1.2,1.3) {$1$};
   \end{tikzpicture}
\end{center}
\caption{Four examples of finite group models taken from \cite{zbMATH06577211}. From left to right, a model with a group of order $4$, $6$, $8$ and $10$. They should be normalised to satisfy the property that $\sum_{(i,j)\in\mathcal S}d_{i,j}=1$. The paper \cite{zbMATH06577211} actually contains infinite families of finite group examples. For example, the third example defines a group model of order $8$ if and only if the associated weights $d_{i,j}$ satisfy $d_{1,0}d_{-1,0}=d_{1,-1}d_{-1,1}\neq 0$. In the examples above, the weights are taken to be rational numbers; note that non-rational weights are also allowed.}
\label{fig:step_sets}
\end{figure}
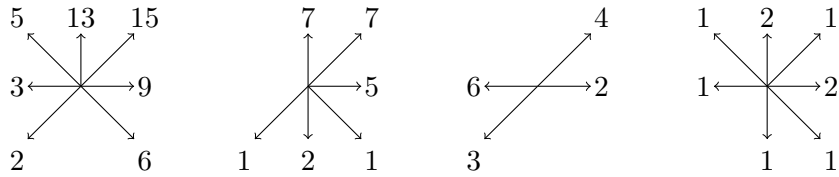

\subsection*{Main results}
Our first and main result is the following; it provides further progress in the classification of lattice walks in the quarter plane. 
\begin{thm}
\label{thm:equi_cond}
The following are equivalent.
\begin{enumerate}
 \item\label{it:1:groupf}The group is finite. 
 \item\label{it:2:seriesx}The series $Q(x,y,t)$ satisfies a (non-trivial)\ linear differential equation with coefficients in $\C(x,y,t)$, in the variable $x$.
  \item\label{it:3:seriesy}The series $Q(x,y,t)$  satisfies a linear differential equation with coefficients in $\C(x,y,t)$, in the variable  $y$.
   \item\label{it:4:seriest}The series $Q(x,y,t)$ satisfies a linear differential equation with coefficients in $\C(x,y,t)$, in the variable  $t$.
\end{enumerate}
\end{thm}
Throughout this paper, although we will often omit the term non-trivial, it is implicitly assumed that all differential equations are non-trivial, i.e., non-zero. 

Note that $x$ and $y$ play a symmetric role so the equivalence between \ref{it:2:seriesx} and \ref{it:3:seriesy} is not so surprising. On the other hand, the equivalence between the $t$ behavior and the other variables behavior is at first sight much more surprising.  Proofs of the equivalence of the first three conditions of Theorem~\ref{thm:equi_cond} have already appeared in the literature for fixed values of $t$. More specifically, if the assumption~\ref{it:1:groupf} on the finiteness of the group holds, the properties \ref{it:2:seriesx} and \ref{it:3:seriesy} are obtained in \cite[Thm~4.1]{dreyfus2019differential}. On the other hand, if the group is infinite, it is shown in \cite{KuRa12} that the series $Q(x,y,t)$ does not satisfy any linear equation in $x$ or in $y$, further assuming that the step sets are unweighted. The previous result is extended to the weighted case in \cite[Thm~8.7]{EP-22}. Compared to the present work, the main difference is that in the previous references all results are obtained with a fixed value of $t$. The main novelty of our work is precisely to be able to deal with this variable $t$, see~\ref{it:4:seriest}, which is very important from a combinatorial point of view. On the other hand, the result of \cite{EP-22} holds not only for the quarter plane, but for any cone obtained as the union of an odd number of quadrants.

Theorem~\ref{thm:equi_cond} will be obtained as a consequence of two further results, namely Theorems~\ref{thm:main_section_DF} and \ref{thm4}.

\begin{thm}
\label{thm:main_section_DF}
Assume the group of the walk is finite. Then $Q(x,y,t)$ in \eqref{eq:def_generating_function} satisfies a linear differential equation with coefficients in $\C(x,y,t)$ in each of its variables. 
\end{thm}

Theorem~\ref{thm:main_section_DF} is proved in Section~\ref{sec:DF_case}. In the unweighted case it is known \cite[Sec.~3]{BMM} that there are exactly $23$ relevant small step quadrant walk models with finite group. Theorem~\ref{thm:main_section_DF} is proved in \cite{BMM} for $22$ of these $23$ models. For the remaining model, known as Gessel's walk, Theorem~\ref{thm:main_section_DF} is proved in \cite{Bostan10} (note that Gessel's model is the unweighted variant of the third example in Figure~\ref{fig:step_sets}). See \cite{Pech} for closely related results. A weaker version of Theorem~\ref{thm:main_section_DF}, when $t$ is fixed in the unit interval $(0,1)$, is proved in \cite{Fayolle-Raschel-10}. Back to the weighted case, a weaker version of Theorem~\ref{thm:main_section_DF} is given in \cite[Thm~4.1]{dreyfus2019differential}. It is proved that if the group is finite, then for any fixed $t\in(0,1)$, $Q(x,y,t)$ satisfies a linear differential equation in $\mathbb C(x,y)$ in each of the two variables $x$ and $y$. In this sense, our Theorem~\ref{thm:main_section_DF} is a refinement of \cite[Thm~4.1]{dreyfus2019differential}, including the last time variable $t$. As we could mention above, there is no exhaustive description of the parameters $d_{i,j}$ that leads to a finite group. In this sense, Theorem~\ref{thm:main_section_DF} holds independently of an explicit description of the model, the finite group hypothesis is sufficient.

Our next result shows that in the infinite group case the generating function does not satisfy any differential equation.
\begin{thm}\label{thm4}
Assume that the group is infinite. Then $Q(x,y,t)$ does not satisfy any (non-trivial)\ linear differential equation with coefficients in $\C(x,y,t)$, in any of its variables.  
\end{thm}
The proof of non-D finiteness in $x$ and $y$ is inspired by the one used in \cite{KuRa12}. On the other hand, the non-D finiteness in the variable $t$ is more subtle and requires a careful study of the rationality of the periods naturally associated with the model. See Section~\ref{sec:infinite_group_case} for more details.

In the next result we give a refinement of Theorem~\ref{thm:main_section_DF} under the additional assumption that the orbit-sum is zero. By definition, denoting the group by $G$ and the signature of a given element $g\in G$ by $\text{sign}(g)$, the orbit-sum of the function $(x,y)\mapsto x y$ is the quantity
\begin{equation}
\label{eq:formal_orbit_sum}
    \mathcal O(x,y) = \sum_{g\in G}\text{sign}(g) g(x y).
\end{equation} 
For example, the orbit-sum is zero for the last two models in Figure~\ref{fig:step_sets}, and non-zero for the first two models.
\begin{thm}
\label{thm1}
Assume that the group of the walk is finite. Then $Q(x,y,t)$ is algebraic over $\C(x,y,t)$ if and only if $\mathcal{O}(x,y)=0$. 
\end{thm}
Note that by Theorem \ref{thm:equi_cond} when the group is infinite, the series $Q(x,y,t)$ is transcendental.
Theorem~\ref{thm1} will be shown in Section~\ref{secalg}.  
In the unweighted case, Theorem~\ref{thm1} follows from \cite{BMM,Bostan10,BeBMRa-21}. See also \cite{Fayolle-Raschel-10}. In the weighted case, a weak variant of Theorem~\ref{thm1} (for fixed $t$)\ is proved in \cite[Thm~4.1]{dreyfus2019differential}. In principle, Tutte's invariants methodology developed in \cite{BeBMRa-21} in the unweighted case also applies to weighted models, and should lead to algebraic results in three variables; however, to our knowledge this has not been worked out. 

In order to present our last main result (Theorem~\ref{thm2}), we need to comment on general techniques to approach Theorems~\ref{thm:main_section_DF} and~\ref{thm1}. Looking at the literature, the techniques used are mostly quite different in the D-finite case (Theorem~\ref{thm:main_section_DF})\ and the algebraic case (Theorem~\ref{thm1}). For example, in the D-finite case, the reference \cite{BMM} uses explicit expressions for the generating function as positive parts of rational functions (see \cite[Prop.~8]{BMM}, see also \cite{Courtiel}), from which the D-finiteness follows from general theoretical arguments. On the other hand, again in \cite{BMM}, explicit algebraic expressions are obtained from a subtle half-orbit-sum technique (see \cite[Sec.~6]{BMM}). Similarly, Tutte's invariant method of \cite{BeBMRa-21} applies only in the zero-orbit-sum case, while the techniques of \cite{Pech} leading to hypergeometric expressions for $Q(x,y,t)$ apply only in the D-finite case; and so on. 

Remarkably, our techniques for proving Theorems~\ref{thm:main_section_DF} and~\ref{thm1} are similar and rely on a two-step approach. First, we will use classical results \cite{FIM17,new,dreyfus2019differential,dreyfus2021differential} which give expressions for the generating functions $Q(x,y,t)$ in terms of elliptic functions (after lifting to the complex plane). Second, we will use a single technical result about elliptic functions, given as Theorem~\ref{thm2}. Although it is a little too technical to be presented in this introduction, we can give the intuitive idea. Theorem~\ref{thm2} can be interpreted as a refinement of one of the most classical statements in elliptic function theory, which asserts that any elliptic function is determined by its poles and the principal parts at its poles, up to some additive constant. More specifically, we will show that controlling the arithmetic nature (i.e., algebraicity or D-finiteness)\ of the poles and principal parts of a given elliptic function yields a global control of the arithmetic nature of the function itself.

It is worth mentioning that there is a result similar to Theorem~\ref{thm:equi_cond} for D-algebraicity, see \cite[Thm~1.1]{dreyfus2021differential}, \cite[Thm~2]{dreyfus2019length}, and \cite[Thm~3.8]{
hardouin2021differentially} for more details.
It is proved in \cite{dreyfus2021differential,dreyfus2019length,hardouin2021differentially} that the following are equivalent:
\begin{itemize}
 \item There exists a decoupling function.
 \item The series $Q(x,y,t)$ satisfies an algebraic differential equation with coefficients in $\C(x,y,t)$, in the variable $x$.
  \item The series $Q(x,y,t)$ satisfies an algebraic differential equation with coefficients in $\C(x,y,t)$, in the variable  $y$.
   \item The series $Q(x,y,t)$ satisfies an algebraic differential equation with coefficients in $\C(x,y,t)$, in the variable  $t$.
\end{itemize}
\section{Preliminary results on the kernel curve and generating functions}
\label{secnot}

\subsection*{Functional equation}
In this section we recall some well-known facts about the generating function counting weighted walks in the quarter plane. We follow the presentation of \cite{dreyfus2018nature}; for more details and precise references we refer to \cite{dreyfus2018nature,dreyfus2021differential}. 

Let $\K=\Q(d_{i,j})$ be the field generated by the weights $d_{i,j}$. The kernel of the walk is the bivariate polynomial defined by 
$
K(x,y,t):=xy (1-t S(x,y)),
$
where $S(x,y)$ denotes the jump polynomial
\begin{equation}
\label{eq:def_A_i}
S(x,y) =\sum_{-1\leq i,j\leq 1} d_{i,j}x^i y^j
 =\frac{A_{-1}(x)}{y} +A_{0}(x)+ A_{1}(x) y
 = \frac{B_{-1}(y)}{x} +B_{0}(y)+ B_{1}(y) x,
\end{equation}
with $A_{i}(x) \in x^{-1}\K[x]$ and $B_{i}(y) \in y^{-1}\K[y]$. 
The kernel plays an important role in the so-called kernel method. 
Define further the sectional generating functions
\begin{equation}
\label{eq:def_F1_F2}
     F^{1}(x,t):= K(x,0,t)Q(x,0,t) \quad\text{and}\quad F^{2}(y, t):= K(0,y, t)Q(0,y, t).
\end{equation}
As first proved in \cite{BMM} in the unweighted framework, and then used e.g.\ in \cite{dreyfus2019differential} in the weighted case, the following functional equation holds:
\begin{lem}
\label{lem11}
The generating function $Q(x,y,t)$  satisfies the functional equation
\begin{equation}
\label{eq:funcequ}
     K(x,y,t)Q(x,y,t)=F^{1}(x,t) +F^{2}(y,t)-K(0,0,t) Q(0,0,t)+xy.
\end{equation}
\end{lem}

\subsection*{Kernel curve}
By definition, the kernel curve $E_{t}$ is the complex affine algebraic curve defined by
\begin{equation*}
   E_t = \bigl\{(x,y) \in \C \times \C \ \vert \ K(x,y,t) = 0\bigr\}.
\end{equation*}
We shall consider a compactification of this curve. We let $\P1(\C)$ be the complex projective line, that is the quotient of $(\C \times \C) \setminus \{(0,0)\}$ by the equivalence relation 
\begin{equation*}
   (x_{0},x_{1}) \sim (x_{0}',x_{1}') \quad \Longleftrightarrow \quad\exists \lambda \in \C^{*},\  (x_{0}',x_{1}') = \lambda (x_{0},x_{1}). 
\end{equation*}
The equivalence class of $(x_{0},x_{1}) \in (\C \times \C) \setminus \{(0,0)\}$ is denoted by $[x_{0}:x_{1}] \in \P1(\C)$. The map 
$
x \mapsto  [x:1]
$ 
embeds $\C$ inside $\P1(\C)$. It is not surjective: its image is $\P1(\C) \setminus \{[1:0]\}$; the missing point $[1:0]$  is usually denoted by $\infty$. 
  Now, we embed $E_{t}$  inside $\P1(\C) \times \P1(\C)$ via  ${(x,y) \mapsto ([x:1],[y:1])}$. 
  
The  kernel curve $\Etproj$ is the closure of this embedding of $E_{t}$.  In other words, $\Etproj $ is the algebraic curve defined by
\begin{equation*}
   \Etproj = \{([x_{0}:x_{1}],[y_{0}:y_{1}]) \in \P1(\C) \times \P1(\C) \ \vert \ \overline{K}(x_0,x_1,y_0,y_1,t) = 0\},
\end{equation*}
where $\overline{K}(x_0,x_1,y_0,y_1,t)$ is the following degree-two homogeneous polynomial
\begin{equation*}
\overline{K}(x_0,x_1,y_0,y_1,t)=x_1^2y_1^2K\left(\frac{x_0}{x_1},\frac{y_0}{y_1},t\right)= x_0x_1y_0y_1 -t \sum_{0\leq i,j\leq 2} d_{i-1,j-1} x_0^{i} x_1^{2-i}y_0^j y_1^{2-j}. 
 \end{equation*}
To simplify our notation, we will denote by $\overline{K}(x,y,t)$ the polynomial $\overline{K}(x_0,x_1,y_0,y_1,t)$. We have basically three different options:
\begin{itemize}
\item The kernel curve is degenerate, in the sense of \cite[Def.~2.2]{dreyfus2021differential}, and in this case the generating series is algebraic over $\C(x,y,t)$, see \cite[Lem.~2.4]{dreyfus2021differential}. 
\item The kernel curve has genus zero and either the generating series  is algebraic over $\C(x,y,t)$, or is differentially transcendental in all its variables, see \cite[Lem.~2.6]{dreyfus2021differential}. 
\item $\Etproj$ is an elliptic curve. 
\end{itemize}
So we only need to focus on the situation where $\Etproj$ is an elliptic curve.  Necessary and sufficient conditions on $d_{i,j}$ to ensure that $\Etproj$ is an elliptic curve are given in \cite[Prop.~2.5]{dreyfus2021differential}. 

In this situation, see \cite[Prop.~2.5]{dreyfus2021differential}, $A_{1}(x),A_{-1}(x),B_{1}(y),B_{-1}(y)$ are not identically zero.

\subsection*{Group of the walk}
For a fixed value of $x$, consider the two roots  $y_{-},y_{+}$ of 
$y\mapsto \overline{K}(x,y,t)$.
Then,  $(x,y_{\pm})\in\Etproj$.  We consider the involutive function
\begin{equation*} 
   \iota_{1} : \Etproj \rightarrow \Etproj
\end{equation*} 
that sends $(x,y_{-})$ to $(x,y_{+})$.  Similarly, we define $\iota_2$ that permutes the $x$-roots.  
Let us finally define
\begin{equation*}
     \sigma=\iota_2 \circ \iota_1. 
\end{equation*}
Let $G$ be the group generated by the involutions $\iota_{1}$ and $\iota_{2}$, and let $G_{t}$ be the specialization of this group for a fixed value of $0<t<1$. 
 
In the unweighted case,  the algebraic nature of the generating  series depends on whether $\sigma$ has finite or infinite order, as implied by a combination of several works \cite{MR2553316,BMM,Bostan10,KuRa12}. More precisely, $G$ is finite if and only if the generating series is D-finite, i.e., it satisfies a non-trivial linear differential equation with coefficients in $\C(x,y,t)$ in each of its three variables. Also note that if $G$ is infinite, then $G_{t}$ can be either finite or infinite. 
 
As in Section~\ref{sec:intro}, see \eqref{eq:formal_orbit_sum}, we define the orbit-sum
    $\mathcal O(x,y) = \sum_{g\in G}\text{sign}(g) g(x y)$,
where in the above, $\text{sign}(g) = 1$ (resp.\ $-1$)\ if the number of elements $\iota_1,\iota_2$ used to write $g$ is even (resp.\ odd).

\subsection*{Parametrization} 
The elliptic curve $\Etproj$ is biholomorphic to $\C/\bigl(\omega_1 (t)\Z  + \omega_2 (t)\Z\bigr)$ for some lattice $\omega_1 (t)\Z+ \omega_2 (t)\Z$ of $\C$  via the $\bigl(\omega_1 (t)\Z + \omega_2 (t)\Z\bigr)$-periodic holomorphic map 
\begin{equation*}
\begin{array}{llll}
\Lambda :& \C& \rightarrow &\overline{E}_t\\
 &\omega &\mapsto& \bigl(x(\omega,t), y(\omega,t)\bigr),
\end{array}
\end{equation*}
where $x$ and $y$ are rational functions of $\wp$ and its derivative $\partial_{\omega}\wp$, and $\wp$ is the Weierstrass function associated with the lattice $\omega_1 (t)\Z + \omega_2 (t)\Z$:
\begin{equation*}
     \wp(\omega ,t)=\frac{1}{\omega^{2}}+ \sum_{(\ell_{1},\ell_{2}) \in \Z^{2}\setminus \{(0,0)\}} \left(\frac{1}{\bigl(\omega +\ell_{1}\omega_{1}(t)+\ell_{2}\omega_{2}(t)\bigr)^{2}} -\frac{1}{\bigl(\ell_{1}\omega_{1}(t)+\ell_{2}\omega_{2}(t)\bigr)^{2}}\right).
\end{equation*}     
The maps $\iota_{1}$, $\iota_{2}$ and $\sigma$ may be lifted to the $\omega$-plane $\C$. We denote these lifts by  $\iup_{1}$, $\iup_{2}$ and $\widetilde{\sigma}$, respectively. So we have the commutative diagrams 
\begin{equation*}
\xymatrix{
    \Etproj  \ar@{->}[r]^{\iota_k} & \Etproj  \\
    \C \ar@{->}[u]^\Lambda \ar@{->}[r]_{\iup_k} & \C \ar@{->}[u]_\Lambda 
  }
  \qquad\qquad\qquad
  \xymatrix{
    \Etproj  \ar@{->}[r]^{\sigma} & \Etproj  \\
\C \ar@{->}[u]^\Lambda \ar@{->}[r]_{\widetilde{\sigma}} & \C \ar@{->}[u]_\Lambda 
  }  
\end{equation*}
The lifted maps are of the form 
\begin{equation*}
   \iup_{1}(\omega)=-\omega, \quad\iup_{2}(\omega)=-\omega+\omega_{3}\quad \text{and} \quad\widetilde{\sigma}(\omega)=\omega+\omega_{3},
\end{equation*}
for some $\omega_{3}(t)\in (0,\omega_2(t))$. 
We now give explicit expressions for the periods $\omega_1(t),\omega_2(t)$,  $\omega_{3}(t)$ and the coordinates $x(\omega,t),y(\omega,t)$. 
For $[x_0:x_1]$ in $\P1(\C)$, we denote by ${\Delta_{1}([x_0:x_1],t)}$ the discriminant of the degree-two homogeneous polynomial $y \mapsto \overline{K}(x_0,x_1,y,1,t)$.
 Let us write 
\begin{equation*}
   \Delta_{1}([x_0:x_1],t)= \displaystyle\sum_{i=0}^{4}\alpha_{i}(t)x_{0}^{i}x_{1}^{4-i}.
\end{equation*}
The discriminant $\Delta_{1}([x_0:x_1],t)$ admits four distinct continuous real roots $a_1(t),\dots,a_{4}(t)$. They are numbered so that the cycle of $\mathbb P_1(\R)$ from $-1$ to $\infty$ and from $-\infty$ to $-1$ crosses them in the order $a_1(t),\ldots , a_4(t)$. One of them, say $a(t)$, is different from $[1:0]$ for all $t\in (0,1)$.  Since the discriminant has coefficients in $\K(t)$,  we deduce that $a(t)$ (and the other $a_i(t)$ as well)\ are algebraic over $\C(t)$.

Similarly, we denote by  $b(t)$ a  continuous real root of $\Delta_{2}([y_0:y_1],t)$,  the discriminant $x \mapsto \overline{K}(x,1,y_0,y_1,t)$, that is not $[1:0]$.  For similar reasons,  $b(t)$ is algebraic over $\C(t)$. 
\begin{prop}
\label{prop:uniformization}
For $i=1,2$, let us set $D_i(\star ,t):=\Delta_{i}([\star:1],t)$. An explicit expression for the periods is given by the elliptic integrals
\begin{equation*}
     \omega_{1} (t)=\mathbf{i}\int_{a_{3}(t)}^{a_{4}(t)} \frac{dx}{\sqrt{\vert D_1(x,t)\vert}}\in \mathbf{i}\R_{>0}\quad\text{and}\quad
\omega_{2}(t)=\int_{a_{4}(t)}^{a_{1}(t)} \frac{dx}{\sqrt{D_1(x,t)}}\in \R_{>0}.
\end{equation*}
An explicit expression of the map $
\Lambda(\omega,t)=(x(\omega,t),y(\omega,t))$
is given by 
\begin{itemize}
\item
 $x(\omega,t)=\left[a(t)+\frac{D'_{1}(a(t),t)}{\wp (\omega,t)-\frac{1}{6}D''_{1}(a(t),t)}:1\right]$;
\vspace{0.2cm}
\item 
$y(\omega,t)=\left[b(t)+\frac{D'_{2}(b(t),t)}{\wp (\omega-\omega_3 (t)/2 ,t)-\frac{1}{6}D''_{2}(b(t),t)}:1\right]$.\vspace{0.2cm}
\end{itemize}
An explicit expression of $\omega_3(t)$  is given by
\begin{equation*}
   \omega_{3}(t)=\int_{x_{\pm}(b_1(t),t)}^{a_1(t)} \frac{dx}{\sqrt{D_1(x,t)}}\in (0,\omega_{2}(t)),
\end{equation*}
and $x_{\pm}(y,t)$ are the two roots of $\overline{K}(x_{\pm}(y,t),y,t)=0$. 
\end{prop}

\subsection*{Analytic continuation of the generating functions}
Let us now focus on the analytic continuation of the generating functions $F^1$ and $F^2$ introduced in \eqref{eq:def_F1_F2}.  Let us fix $t\in (0,1)$.
The generating function $Q(x,y, t)$ converges for $\vert x\vert,\vert y\vert<1$. The projection of this set  inside $\mathbb{P}^{1}(\C)\times \mathbb{P}^{1}(\C)$ has a non-empty intersection with the kernel curve $\Etproj$. In virtue of~\eqref{eq:funcequ}, we then find for $\vert x\vert,\vert y\vert<1$ and $(x,y)\in \Etproj$,
\begin{equation*}
   F^{1}(x, t) +F^{2}(y, t)-K(0,0, t) Q(0,0, t)+xy=0.
\end{equation*}
Since the  series $F^{1}(x, t)$ and  $F^{2}(y, t)$ converge for $\vert x\vert$ and $\vert y\vert<1$ respectively,
we can continue $F^{1}(x, t)$  for $(x,y)\in \Etproj$ and $\vert y\vert<1$ with the formula: 
\begin{equation*}
   F^{1}(x, t) =-F^{2}(y, t)+K(0,0, t) Q(0,0, t)-xy.
\end{equation*}
We continue $F^{2}(y,t)$  for $(x,y)\in \Etproj$ and $\vert x\vert<1$ similarly. 
\begin{lem}[Lemma 27 in \cite{dreyfus2021differential}]
\label{lem:existence_properties_O}
There exists a connected set $\mathcal{O}\subset \C$ such that 
\begin{enumerate}
   \item\label{it:O1}$\Lambda (\mathcal{O})=\{(x,y)\in \Etproj : \vert x\vert<1 \hbox{ or } \vert y\vert<1 \}$;
   \item\label{it:O2}$\widetilde{\sigma}^{-1}(\mathcal{O})\cap \mathcal{O}$ is non-empty;
   \item\label{it:O3}$\displaystyle \bigcup_{\ell\in \Z}\widetilde{\sigma}^{\ell}(\mathcal{O})=\C$.
\end{enumerate}
\end{lem}
Furthermore, there exist functions $\omega\mapsto r_x(\omega,t)$ and $\omega\mapsto r_y(\omega,t)$ which are meromorphic on the connected domain $\mathcal{O}$ and such that
\begin{equation}
\label{eq:liftings_GF}
   \left\{\begin{array}{rclcl}
   \rx(\omega, t)&=&F^{1}(x(\omega, t),t)& =&K(x(\omega, t),0,t)Q(x(\omega, t),0,t),\smallskip\\
   \ry(\omega, t)&=&F^{2}(y(\omega, t),t)& =&K(0,y(\omega, t),t)Q(0,y(\omega, t),t).
   \end{array}\right.
\end{equation}
With~\ref{it:O2} and~\ref{it:O3} of Lemma~\ref{lem:existence_properties_O}, we extend $\rx(\omega, t)$ and $\ry(\omega, t)$ as meromorphic functions on $\C$, where they satisfy the equations
\begin{align}
     \rx(\omega+\omega_{3}(t),t) & =\rx(\omega,t)+b_{x}(\omega,t) ,\label{eq:omega_3_per_rx} \\
     \rx(\omega+\omega_{1}(t),t)&=  \rx(\omega,t),\label{eq:omega_1_per_rx}    \\ 
     \ry(\omega+\omega_{3}(t),t)   &=\ry(\omega,t)+b_{y}(\omega,t), \label{eq:omega_3_per_ry} \\ 
     \ry(\omega+\omega_{1}(t),t)&=  \ry(\omega,t),\nonumber\\
     \rx (\omega,t)+\ry(\omega,t)&=K(0,0,t)Q(0,0,t)-x(\omega,t)y(\omega,t),\label{eq:lifting_func_eq}
\end{align}
where
\begin{equation}
\label{eq:def_bx_by}
   b_{x}(\omega,t) =y(-\omega,t)\bigl(x(\omega,t)-x(\omega+\omega_{3}(t),t)\bigr) \hbox{ and }b_{y}(\omega,t)=x(\omega,t)\bigl(y(\omega,t)-y(-\omega,t)\bigr).
\end{equation}

\section{Algebraic nature of elliptic functions depending on a parameter}
\label{seccrit}

In elliptic functions theory, it is well known that any elliptic function is determined by its poles and the principal parts at its poles, up to some additive constant, see e.g.\ \cite[Thm~3.14.4]{JoSi-87}. In this section, we will prove a refinement of this result, see Theorem~\ref{thm2} below. More specifically, we will show that controlling the arithmetic nature (i.e., algebraicity, D-finiteness or differential algebraicity)\ of the poles and principal parts of a given elliptic function yields a global control of the arithmetic nature of the function itself. Theorem~\ref{thm2} is not only interesting in itself, it is also, as we shall see, perfectly adapted to the setting of random walks in the quadrant, in the sense that one single statement will capture various situations occurring in the theory (algebraic, D-finite and D-algebraic generating functions).

Let $\mathcal{P}$ be the field of germs of meromorphic Puiseux series at a given point   $(\omega_0,t_0)\in \C \times (0,1)$.  Let
\begin{equation*}
   \cx := \C(x(\omega,t),t).
\end{equation*}
Let $\mathbf{k}$ be an algebraically closed field and $\mathbf{R}$ be a ring such that
 \begin{equation*}
{\cx}\subset\mathbf{k}\subset \mathbf{R}\subset \mathcal{P}. 
\end{equation*}    
   
The associated constant fields are denoted by $\cx_t$, $\mathbf{k}_t$ and $\mathbf{R}_t$:
\begin{equation*}
  \quad \cx_t = \mathbb C(t),\quad \mathbf{k}_t=\{a\in \mathbf{k}: \partial_{\omega} (a)=0 \}\quad \text{and}\quad  \mathbf{R}_t=\{a\in \mathbf{R}: \partial_{\omega} (a)=0 \} . 
\end{equation*}
 The aim of this article is to determine whether certain functions lie in $\mathbf{R}$ in the following situations:
\begin{itemize}
\item The {\em algebraic} case: $\mathbf{k}=\mathbf{R}=\overline{{\cx}}$;
\item The {\em D-finite} case: $\mathbf{k}=\overline{{\cx}}$,  $\mathbf{R}$ is the ring of  functions that are solutions to linear $\partial_{\omega}$ and $\partial_t$-differential equations with coefficients in $\overline{{\cx}}$;
\item The {\em D-algebraic} case: $\mathbf{k}=\mathbf{R}$ is the field of functions that are solutions to algebraic $\partial_{\omega}$ and $\partial_t$-differential equations with coefficients in $\overline{{\cx}}$. 
\end{itemize}
The purpose of defining $\mathbf{k}$ separately from $\mathbf{R}$ is that some functions we will be considering will sometimes involve ratios, so to prove that the ratio lies in $\mathbf{R}$ it will be sufficient to prove that the numerator and denominator both lie in $\mathbf{k}$ (or that the numerator lies in $\mathbf{R}$ while the denominator lies in $\mathbf{k}$).

 Let  $(\wp)_{t\in (0,1)}$ be a family of Weierstrass functions. The function $\wp$ depends on two variables $(\omega,t)$; let $g_2,g_3$ be the invariants (which depend upon $t$). Let us now specify the $t$-dependency of these functions.

\begin{assu}
\label{assu1}
$\wp\in \mathbf{k}$ and $g_2,g_3\in \mathbf{k}_t$.
\end{assu}

\begin{lem}
\label{rem1}
Let $k\geq 0$. Then $\partial_{\omega}^k\wp \in \mathbf{k}$. Moreover, if $\wp(a(t),t)\in\mathbf{k}_t$ for some quantity $a(t)$, then $\partial_{\omega}^{k}\wp(a(t),t)\in\mathbf{k}_t$.
\end{lem}

\begin{proof}
The first item follows from Assumption~\ref{assu1} together with the classical differential equation satisfied by $\wp$:
\begin{equation}
\label{eq:classical_diffeq_P}
   (\partial_{\omega}\wp)^2=4\wp^3-g_{2}\wp-g_{3}.
\end{equation}
We prove the second statement by induction, the case $k=0$ is trivial and the case  $k=1$ following from Assumption~\ref{assu1} and \eqref{eq:classical_diffeq_P}. For the inductive step, we take the $k$th derivative of both sides of \eqref{eq:classical_diffeq_P} and rearrange the result so as to express $\partial_{\omega}^{k+1}\wp(a(t),t)$ as a rational function of $g_{2},g_{3}$ and the lower order derivatives $\partial_{\omega}^{j}\wp(a(t),t)$, $0\leq j\leq k$. This implies that if these lower order derivatives belong to $\mathbf{k}_t$, then so does $\partial_{\omega}^{k+1}\wp(a(t),t)$. 
\end{proof}

Lemma~\ref{rem1} suggests to introduce the set $X_{t}$ of functions $b:(0,1)\to\mathbb{C}$, defined by
\begin{equation}
\label{eq:def_X_t}
   X_{t}:=\bigl\{b(t):\wp(b(t),t)\in\mathbf{k}_t\cup\{\infty\}\bigr\}.
\end{equation}

Recall that a meromorphic function admits a Laurent series expansion $\sum_{\ell=\nu}^{\infty} a_{\ell}(\omega-a)^{\ell}$ at any given point $a\in \C$. By definition, its principal part at $a$ is the rational function $\sum_{\ell=\nu}^{-1} a_{\ell}(\omega-a)^{\ell}$ (with the convention that the sum is $0$ if $\nu\geq 0$). The coefficients of the principal part at $a$ are the  complex numbers $a_{\nu},\dots,a_{-1}$. Given a meromorphic function on $\C$, the coefficients of its principal parts are the collection of all coefficients of the principal parts at its poles. If this function is elliptic, the coefficients of its principal parts form a finite set.

The following theorem is one of our main technical results. 

\begin{thm}
\label{thm2}
For $t\in (0,1)$, let $\omega \mapsto f(\omega,t)$ be a meromorphic function on $\C$.  Let us assume that Assumption~\ref{assu1} holds, and that:  
\begin{enumerate}[label={\rm(\roman*)},ref={\rm(\roman*)}]
   \item For all $t\in (0,1)$, $\omega\mapsto f(\omega,t)\in \C(\wp,\partial_{\omega}\wp)$; \label{thm2:condition_elliptic} 
   \item The  poles of $\omega\mapsto f(\omega,t)$ belong to $X_t$; \label{thm2:condition_nicepoles}
   \item The coefficients of the principal parts  of $\omega\mapsto f(\omega,t)$ belong to $\mathbf{R}_t$; \label{thm2:condition_principalcoefficients}
   \item There exists $a(t)\in X_t$ such that  $f(a(t),t)\in \mathbf{R}_t$. \label{thm2:condition_existsnicea}
\end{enumerate}
Then $f(\omega,t)\in \mathbf{R}$.
\end{thm}
\begin{proof}
Let us begin with the case where for all $t\in (0,1)$, $f(\omega,t)$ is an even function of $\omega$. Any even elliptic function being a rational function of the Weierstrass function $\wp$ with the same periods (see, e.g., \cite[p.~44]{akhiezer1990elements}), we have a decomposition of the form 
\begin{equation}
\label{eq6}
   f(\omega,t)={c}(t)+ \sum_{i= 1}^{n_{\infty}}a_{i,\infty}(t)\wp (\omega,t)^{i}+ \sum_{j} \displaystyle \sum_{i= 1}^{n_j} \frac{a_{i,j}(t)}{\bigl(\wp (\omega,t)-\wp (b_j(t),t)\bigr)^{i} },
\end{equation}
where  
\begin{itemize}
   \item ${c}(t),a_{i,\infty}(t), a_{i,j}(t)$ are functions that do not depend upon $\omega$;
   \item $b_j(t)\in X_t \setminus \{\infty \}$ are the poles of $\omega\mapsto f(\omega,t)$;
   \item the sum $\sum_j$ is finite.
\end{itemize}

We first prove that for all $1\leq i \leq n_\infty$, $a_{i ,\infty}(t)\in \mathbf{R}_t$. To that purpose, write the series expansion of the $\wp$-Weierstrass function at $\omega=0$:
\begin{equation}
\label{eq:expansion_0_wp}
   \wp(\omega,t)=\frac{1}{\omega^2} + \sum_{k=2}^\infty f_k(t)\omega^{k}.
\end{equation}
If $n_\infty =0$ there is nothing to prove. Assume that  $n_\infty \neq 0$.
Plugging \eqref{eq:expansion_0_wp} in the identity \eqref{eq6}, we obtain that  $f(\omega,t)$ is equivalent to $a_{n_{\infty},\infty}(t)\omega^{-2n_{\infty}}$ as $\omega$ goes to $0$. Therefore, using our assumption~\ref{thm2:condition_principalcoefficients}, it implies that $a_{n_{\infty},\infty}(t)$ is in $\mathbf{R}_t$. 

We now look at cases $1\leq i < n_\infty$. It is well known that the coefficients $f_k(t)$ in \eqref{eq:expansion_0_wp} are all polynomials with rational coefficients in the invariants $g_{2}$ and $g_{3}$, see for instance the proof of Theorem~3.16.2 in \cite{JoSi-87}.  Accordingly, using Lemma~\ref{rem1}, all $f_k(t)$ belong to $\mathbf{k}_t$. This shows that for $1\leq i < n_\infty$, the coefficient of $\omega^{-2i}$ in the series expansion of $f(\omega,t)$ in \eqref{eq6} at $\omega= 0$ has the form $a_{i,\infty}(t)+f_{i,\infty}$, where $f_{i,\infty}\in\mathbf{k}_t( a_{i+1,\infty}(t),\dots ,a_{n_\infty ,\infty}(t))$. By condition~\ref{thm2:condition_principalcoefficients}, we know that the previous coefficients all lie in $\mathbf{R}_t$, so it follows by a decreasing induction  that for all $1\leq i \leq n_\infty$,   $a_{i ,\infty}(t)\in \mathbf{R}_t$.

We now prove that for all $j$ and all $1\leq i \leq n_j$, $a_{i ,j}(t)\in \mathbf{R}_t$. The reasoning is very similar to the previous case ($j=\infty$), though it has some particularities, which lead us to present the argument in detail.

For a fixed  $j$,  let $\ell\geq 1$ be minimal such that $\partial_{\omega}^{\ell}\wp(b_j(t),t)$ is not identically zero (by classical properties of Weierstrass functions, $\ell\in\{1,2\}$), so that we may write 
\begin{equation}
\label{eq:one-term-asymp}
   \frac{1}{\bigl(\wp (\omega,t)-\wp(b_j(t),t)\bigr)^{i} }=\sum_{k=1}^{\ell i}\frac{c_{j,k}(t)}{(\omega-b_j(t))^k} +O(1),
\end{equation}
where the $c_{j,k}(t)$ are such that
\begin{equation*}
   c_{j,k}(t)\in\mathbb Q\bigl(\partial_{\omega}\wp(b_j(t),t),\ldots,\partial^{\ell i}_{\omega}\wp(b_j(t),t)\bigr)\subset \mathbf{R}_t.
\end{equation*}
The order of the pole of $f(\omega,t)$ at $b_j(t)$ is $\ell n_j$, see \eqref{eq6}, so a one-term asymptotic expansion \eqref{eq:one-term-asymp} gives the dominant term in the principal part of $f$ at $b_j(t)$. As a consequence, if $i=n_{j}$ in \eqref{eq6}, the coefficients of the principal parts of $\omega\mapsto f(\omega,t)$ being in $\mathbf{R}_t$ (Assumption~\ref{thm2:condition_principalcoefficients} of Theorem~\ref{thm2}), and since by Lemma~\ref{rem1},  $\partial_{\omega}^{k}\wp(b_j(t),t)\in\mathbf{k}_t$, we deduce that $a_{n_j,j}\in \mathbf{R}_t$.  

If now $1\leq i < n_j$, one needs to go further and to look at the complete expansion \eqref{eq:one-term-asymp}. The coefficient of the term $(\omega- b_j(t))^{-\ell i}$ in the Laurent expansion of $f(\omega,t)$ admits the form $c_{j,\ell i}(t)+f_{i,j}$,  where $f_{i,j}\in\mathbf{k}_t ( a_{i+1,j}(t),\dots ,a_{n_j ,j}(t))$. By a decreasing induction, we deduce that for all $1\leq i \leq n_j$,   $a_{i ,j}(t)\in \mathbf{R}_t$.

Recall that by condition~\ref{thm2:condition_existsnicea}, there is some function $a(t)\in X_t$ such that $f(a(t),t)\in \mathbf{R}_t$. If $a(t)$ happens to be a pole of $\omega \mapsto \wp(\omega,t)$, substituting $\omega=a(t)$ into \eqref{eq6} then yields 
$a_{1,\infty}(t)=\dots =a_{n_{\infty},\infty}(t)=0$ and ${c}(t)\in \mathbf{R}_t$.  
When $a(t)$ is not a pole, then $\wp (a(t),t)\in \mathbf{k}_t$ and substituting $\omega=a(t)$ into \eqref{eq6} gives
\begin{equation*}
   {c}(t)+\sum_{i= 1}^{n_{\infty}}a_{i,\infty}(t)\wp (a(t),t)^{i}+ \sum_{j} \sum_{i= 1}^{n_j} \frac{a_{i,j}(t)}{\bigl(\wp (a(t),t)-\wp(b_j(t),t)\bigr)^{i} }\in \mathbf{R}_t.
\end{equation*}
Moreover, every term in this expression apart from ${c}(t)$ belongs to the ring $\mathbf{R}_t$, so we must have ${c}(t)\in\mathbf{R}_t$.  In both cases, we have proved that ${c}(t)\in\mathbf{R}_t$.

Summing up, we have shown that every term in the right-hand side of \eqref{eq6} belongs to $\mathbf{R}_t (\wp(\omega,t))$.  Since $\wp(\omega,t)\in \mathbf{k}\subset \mathbf{R}$ by Assumption~\ref{assu1}, we find that $f(\omega,t)$ belongs to $\mathbf{R}$. This concludes the proof in the  even case.  
 
We now consider the case where $\omega\mapsto f(\omega,t)$ is odd for all $t\in (0,1)$. 
Then the function $ \partial_{\omega}\wp(\omega,t)^{-1} f(\omega,t)$ is even and may be written in the same form as \eqref{eq6}, so that we have 
\begin{equation*}
f(\omega,t)=\sum_{i=0}^{n_{\infty}}a_{i,\infty}(t)\partial_{\omega}\wp(\omega,t)\wp (\omega,t)^{i}+ \sum_{j} \displaystyle \sum_{i= 1}^{n_j} \frac{a_{i,j}(t)\partial_{\omega}\wp(\omega,t)}{\bigl(\wp (\omega,t)-\wp (b_j(t),t)\bigr)^{i} },
\end{equation*}
where the $i=0$ term in the first sum above replaces the ${c}(t)$ term occurring in \eqref{eq6}. Similar computations as in the even case give that $f(\omega,t)$ belongs to $\mathbf{R}$.
  
We finally consider the general case. For any integer $n$, $\bigl(\wp(\omega,t)-\wp (a(t),t)\bigr)^{n}\in \mathbf{k}$. Then,  $f(\omega,t)\in \mathbf{R}$ if and only if $f(\omega,t)\bigl(\wp(\omega,t)-\wp (a(t),t)\bigr)^{n}\in \mathbf{R}$. So without loss of generality, we may reduce to the case where $f( a(t),t)=f( -a(t),t)=0$.

Classically, we write 
\begin{equation*}
   f(\omega,t)=f_{+}(\omega,t)+f_{-}(\omega,t),
\end{equation*}
  where $f_{+}(\omega,t)=\frac{f(\omega,t)+f(-\omega,t)}{2}$ is even and $f_{-}(\omega,t)=\frac{f(\omega,t)-f(-\omega,t)}{2}$ is odd.
  Given that conditions~\ref{thm2:condition_elliptic}--\ref{thm2:condition_principalcoefficients} hold for $f(\omega,t)$, they hold as well for $f(-\omega,t)$, and therefore for $f_{+}(\omega,t)$ and $f_{-}(\omega,t)$. With $f(\pm a(t),t)=0$, we find $f_{\pm}( a(t),t)=0$. Then our condition~\ref{thm2:condition_existsnicea} also holds.
From the even and odd cases of this theorem, we deduce that $f_{\pm }(\omega,t)\in\mathbf{R}$, so $f(\omega,t)=f_{+}(\omega,t)+f_{-}(\omega,t)\in\mathbf{R}$, as required.
 This completes the proof.
\end{proof}

\section{Algebraic case and transformation theory of elliptic functions}
\label{secalg}

The main objective of this section is to prove Theorem~\ref{thm1}.
This result states that, under the assumption that the group of the walk is finite, $Q(x,y,t)$ is algebraic over $\C(x,y,t)$ if and only if the orbit-sum $\mathcal{O}(x,y)$ defined in \eqref{eq:formal_orbit_sum} is identically zero. 

\subsection{Strategy of proof}

It will be convenient to write the orbit-sum directly as a function of $\omega$ and $t$, for any $\omega\in\mathbb C$ and $t\in (0,1)$:
\begin{equation}
\label{eq:orbit_sum_Ox}
   \mathcal{O}_x(\omega,t) = \mathcal O\bigl(x(\omega,t),y(\omega,t)\bigr).
\end{equation}
Obviously, the hypothesis on the orbit-sum implies that $\mathcal{O}_x(\omega,t)=0$. 
Note that \cite[Thm~4.1]{dreyfus2019differential} proves that for $t_0\in (0,1)$ fixed, $Q(x,y,t_0)$ is algebraic over $\C(x,y)$ if and only if  $\mathcal{O}_x(\omega,t_0)$ is identically zero. Hence, if $Q(x,y,t)$ is algebraic over $\C(x,y,t)$ then $\mathcal{O}_x(\omega,t)=0$ for all $t\in (0,1)$. The following lemma proves that the formal orbit-sum $\mathcal{O}(x,y)$ in \eqref{eq:formal_orbit_sum} is zero as well, proving the ``if part'' in Theorem~\ref{thm1}. In the remaining of Section~\ref{secalg}, we will thus concentrate on the ``only if'' part.

\begin{lem}\label{lem9}
Assume that for all $t\in (0,1)$, the orbit-sum $\mathcal{O}(x,y)$ vanishes on $\Etproj$. Then $\mathcal{O}(x,y)$ is identically zero on $\mathbb C^2$. 
\end{lem}

\begin{proof}
Let us fix $t\in (0,1)$ and $x\in \C$. Let $y_{\pm}(x,t)$ be such that $\overline{K}(x,y_{\pm}(x,t),t)=0$. Note that  the roots  $y_{\pm}(x,t)$ depend on $t$. We may restrict $(0,1)$ to an interval $U$, and without loss of generality reduce to the case where for such $x$ fixed,  $y_{\pm}(x,t)$ depends continuously upon $t\in U$. In particular, $\{y_{+}(x,t) \vert t\in U\}$ admits an infinite number of values.
  We now use the fact that  $y\mapsto \mathcal{O}(x,y)$ is a rational function. Since for all $t\in U$, $\mathcal{O}(x,y_{+}(x,t))=0$ we deduce that $y\mapsto \mathcal{O}(x,y)$ has an infinite number of roots, it is therefore $0$. Since $x$ is arbitrary, we deduce that $\mathcal{O}(x,y)$ is zero.
\end{proof}

Recall that the period $\omega_{1}$ is purely imaginary, while $\omega_{2}$ and $\omega_{3}$ are positive real numbers, see Section~\ref{secnot}. Moreover, we consider the transformation $\widetilde{\sigma}$ from Section~\ref{secnot} satisfying $ \widetilde{\sigma}(\omega)=\omega+\omega_{3}$. 
Since the group is finite, there exists a non-zero $\ell\in \N$ such that $\widetilde{\sigma}^{\ell}$  leaves invariant the lattice $\Lambda=\omega_{1}\mathbb{Z}+\omega_{2}\mathbb{Z}$. Hence $\ell\omega_3\in \Lambda$, so there exists a positive number $k$ such that 
\begin{equation}
\label{eq:ratio_periods}
   \frac{\omega_3}{\omega_2}=\frac{k}{\ell}.
\end{equation}
We may further assume that $k$ and $\ell$ are coprime.

Let us iterate \eqref{eq:omega_3_per_rx} and \eqref{eq:omega_3_per_ry} to deduce that
\begin{align}
\label{eqrx_transform_finite_group}    
   \rx(\omega+\ell\omega_{3}(t),t)  &=\rx(\omega,t)+\mathcal{O}_{x}(\omega,t),\\
   \nonumber 
   \ry(\omega+\ell\omega_{3}(t),t)  &=\ry(\omega,t)+\mathcal{O}_{y}(\omega,t),\end{align}
where
\begin{equation}
\label{eq:orbit_sums}   
   \mathcal{O}_{x}(\omega,t)=\sum_{j=0}^{\ell-1} b_{x}(\omega+j\omega_3 (t),t)\quad  \text{and} \quad \mathcal{O}_{y}(\omega,t)=\sum_{j=0}^{\ell-1} b_{y}(\omega+j\omega_3 (t),t)
\end{equation}
are the orbit-sum \eqref{eq:orbit_sum_Ox} and its $y$-analogue. 
It is shown in the proof of \cite[Thm~4.1]{dreyfus2019differential} that  $\mathcal{O}_{x}(\omega,t)=- \mathcal{O}_{y}(\omega,t)$, so the orbit-sums $\mathcal{O}_{x}$ and $\mathcal{O}_{y}$ are simultaneously zero or non-zero.

In this section we prove that if the orbit-sum is zero, then $Q(x,y,t)$ is algebraic over $\C(x,y,t)$.  It suffices to show that if for all $t\in (0,1)$, $\mathcal{O}_{x}(\omega,t)= \mathcal{O}_{y}(\omega,t)=0$ then $Q(x,y,t)$ is algebraic over $\C(x,y,t)$.
Using \eqref{eq:funcequ}, it suffices to show that $F^1$ and $F^2$ (defined in \eqref{eq:def_F1_F2})
are algebraic over $\C(x,t)$ and $\C(y,t)$, respectively (since $K(0,0,t)Q(0,0,t)=F^{1}(0,t)=F^{2}(0,t)$). We will only prove the result for $F^{1}(x,t)$, as the other case would be derived similarly, by symmetry.

For $j,k\in \N$ and non-zero, let $\wp^{(j,k)} (\omega,t)$ be the Weierstrass function with periods $(j\omega_1 (t),k\omega_2 (t))$, with invariants denoted by  $g_2^{(j,k)}$ and $g_3^{(j,k)}$. 
 
Assume that for all $t\in (0,1)$, $\mathcal{O}_{x}(\omega,t)= \mathcal{O}_{y}(\omega,t)=0$.
By \eqref{eqrx_transform_finite_group}, the analytic continuation  $\rx (\omega,t)$ of $F^{1}(x(\omega,t),t)$ is $(\omega_1 (t),k\omega_2 (t))$-elliptic,  with $\ell\omega_3=k\omega_2$.  We want to apply Theorem~\ref{thm2} with $\mathbf{k}=\mathbf{R}=\overline{{\cx}}$ and $f=r_{x}$. 

One key point (which is independent of the orbit-sum being zero and will be true in the D-finite case presented in Section~\ref{sec:DF_case} as well) is to prove that Assumption~\ref{assu1} is satisfied, namely: 
\begin{thm}
\label{thm3}
Assume that the group of the walk is finite.  Then  $\wp^{(1,k)}$ is algebraic over ${\cx}$ and $g_2^{(1,k)},g_3^{(1,k)}$ are algebraic over $\mathbb C(t)$.
\end{thm}

\subsection{Preliminary results on transformed Weierstrass functions}
\label{subsec:prel}

Theorem~\ref{thm3} will be obtained as a direct consequence of Lemmas~\ref{lem7} and~\ref{lem8} below. In order to prove these lemmas, we will first prove several intermediate results, and will show successively that $g_2^{(1,1)}, g_3^{(1,1)}\in\overline{\mathbb C(t)}$ (see Lemma~\ref{lem4}), $\wp^{(1,1)} (\omega,t)\in\overline{{\cx}}$ (Lemma~\ref{lem5}), $g_2^{(k,k)}, g_3^{(k,k)}\in\overline{\mathbb C(t)}$ (Lemma~\ref{lem2}) and finally $\wp^{(k,k)} (\omega,t)\in\overline{{\cx}}$ (Lemma~\ref{lem6}). It is worth mentioning that while some of the forthcoming lemmas are specific to our context, some others are very general statements on elliptic functions. Note that in Section~\ref{subsec:prel} we do not make the assumption that the orbit-sum is $0$ (nor that the group is finite), so we will be able to reuse these results in Section \ref{sec:DF_case}.

\begin{lem}
\label{lem4}
The invariants $g_2^{(1,1)}$ and $g_3^{(1,1)}$ 	are algebraic over $\C(t)$. 
\end{lem}

\begin{proof}
 See  \cite[(2.7) and (2.8)]{dreyfus2019differential}.
\end{proof}

\begin{lem}
\label{lem5}
The elliptic functions $\wp^{(1,1)} (\omega,t)$ and $\partial_{\omega}\wp^{(1,1)} (\omega,t)$ are algebraic over  $\cx$. 
\end{lem}

\begin{proof}
By  Proposition~\ref{prop:uniformization},  there exists a rational function  $f$ with coefficients in $\overline{\C(t)}$ such that $\wp^{(1,1)} (\omega,t) =f(x(\omega,t))$, so $\wp^{(1,1)} $  is algebraic over ${\cx}$. By  Lemma~\ref{lem4} and Equation~\eqref{eq:classical_diffeq_P}, it follows that $\partial_{\omega}\wp^{(1,1)} (\omega,t)$ is algebraic over ${\cx}$.
\end{proof}

\begin{lem}
\label{lem:deriv}
For any function $h(\omega,t)$ algebraic over ${\cx}$, the same holds for $\partial_{\omega}h(\omega,t)$.
\end{lem}

\begin{proof}
Assume $h$ is non-constant in $\omega$, otherwise the statement is clear. Let $P(h,x,t)$ be a non-zero polynomial satisfying $P\bigl(h(\omega,t),x(\omega,t),t\bigr)=0$, of minimal degree in its first variable. Taking the derivative with respect to $\omega$ yields
\begin{equation*}
   \partial_{\omega}h(\omega,t)P_{h}\bigl(h(\omega,t),x(\omega,t),t\bigr)+\partial_{\omega}x(\omega,t)P_{x}\bigl(h(\omega,t),x(\omega,t),t\bigr)=0,
\end{equation*}
with the notation $P_{\star}=\partial_{\star} P$.
By minimality of $P$, $P_{h}$ is non-zero. Hence, $\frac{\partial_{\omega}h(\omega,t)}{\partial_{\omega}x(\omega,t)}$ is a rational function of $h(\omega,t)$, $x(\omega,t)$ and $t$, so it is algebraic over ${\cx}$. Applying this result to the specific case $h= \wp^{(1,1)}$, we obtain that $\frac{\partial_{\omega}\wp^{(1,1)}(\omega,t)}{\partial_{\omega}x(\omega,t)}$ is algebraic over ${\cx}$. Along with Lemma~\ref{lem5}, this shows that $\partial_{\omega}x(\omega,t)$ is algebraic over ${\cx}$, and so $\partial_{\omega}h(\omega,t)$ itself is algebraic over ${\cx}$. 
\end{proof}

\begin{lem}
\label{lem2}
We have $k^4 g_2^{(k,k)}=g_2^{(1,1)}$ and $k^6 g_3^{(k,k)}=g_3^{(1,1)}$. The invariants  $g_2^{(k,k)}$ and $g_3^{(k,k)}$ are thus algebraic over $\C(t)$ by Lemma~\ref{lem4}. 
\end{lem}

\begin{proof}
Let $\Lambda$ be the lattice $\omega_1 \Z+\omega_2 \Z$.  We have 
\begin{equation*}
   g_2^{(k,k)}=60 \sum_{\lambda \in k\Lambda \setminus \{(0,0)\}}\frac{1}{\lambda^4}= 60 \sum_{\lambda \in\Lambda \setminus \{(0,0)\}}\frac{1}{(k\lambda)^4}=\frac{g_2^{(1,1)}}{k^{4}}.
\end{equation*}
The proof for  $g_3^{(k,k)}$ is similar.  
\end{proof}

We will use the following classical lemma several times. 
\begin{lem}
\label{lem:Wal}
For any integers $j_{1},j_{2}\geq 1$ and any $\lambda\in\mathbb{Q}\setminus\{0\}$, the function $\omega\mapsto \wp^{(j_{1},j_{2})}(\lambda \omega,t)$ is algebraic over $\C\bigl(g_2^{(j_{1},j_{2})},g_3^{(j_{1},j_{2})},\wp^{(j_{1},j_{2})}(\omega,t)\bigr)$.
\end{lem}

Let us first recall three identities satisfied by the Weierstrass elliptic function, namely, a second order differential equation \cite[(3.16.3)]{JoSi-87}, the addition formula \cite[(3.17.9)]{JoSi-87} and the duplication formula \cite[(3.17.10)]{JoSi-87}:
\begin{align}
\label{eq:classical_diffeq_P_2}
   \partial_{\omega}^2 \wp(\omega)&=6\wp^2(\omega)-\frac{g_2}{2},\\
   \wp(m \omega )+\wp(n \omega)+\wp((m+n) \omega )
&=\frac{1}{4}\left(\frac{\partial_{\omega}\wp(m \omega)-\partial_{\omega}\wp(n \omega)}{\wp(m \omega )-\wp(n \omega)}\right)^2,\label{eq:addition_formula}\\
   \wp(2\omega)+2\wp(\omega)
&=\frac{1}{4}\left(\frac{\partial_{\omega}^2 \wp(\omega)}{\partial_{\omega}\wp(\omega)}\right)^2.\label{eq:duplication_formula}
\end{align}

\begin{proof}[Proof of Lemma~\ref{lem:Wal}]
The formula \eqref{eq:duplication_formula} implies that $\wp^{(j_{1},j_{2})}(2\omega,t)$ is algebraically related to  $\wp^{(j_{1},j_{2})}(\omega,t)$ over the 
field $\C\bigl(\partial_{\omega} \wp^{(j_{1},j_{2})}(\omega,t), \partial_{\omega}^2 \wp^{(j_{1},j_{2})}(\omega,t)\bigr)$. 
Using the differential equations \eqref{eq:classical_diffeq_P} and \eqref{eq:classical_diffeq_P_2} satisfied by  $\wp^{(j_{1},j_{2})}$, we obtain that the first and second derivatives of $\wp^{(j_{1},j_{2})}(\omega,t)$ are algebraically related to  $ \wp^{(j_{1},j_{2})}(\omega,t)$ over $\C(g_2^{(j_{1},j_{2})},g_3^{(j_{1},j_{2})})$. As a consequence, Lemma~\ref{lem:Wal} is proved for $\lambda=2$. 

We now prove by induction that for all integers $m\geq 1$, $\wp^{(j_{1},j_{2})}(m\omega,t)$    is algebraically related to $\wp^{(j_{1},j_{2})}(\omega,t)$ over $\C\bigl(g_2^{(j_{1},j_{2})},g_3^{(j_{1},j_{2})}\bigr)$. To that purpose, use the addition formula \eqref{eq:addition_formula} with $n=1$ to deduce that if 
 $\wp^{(j_{1},j_{2})}(m\omega,t)$  with $m>1$ is algebraically related to $\wp^{(j_{1},j_{2})}(\omega,t)$ over $\C\bigl(g_2^{(j_{1},j_{2})},g_3^{(j_{1},j_{2})}\bigr)$, the same holds for 
  $\wp^{(j_{1},j_{2})}((m+1)\omega,t)$.  Note that $\wp^{(j_{1},j_{2})}(-m\omega,t)=\wp^{(j_{1},j_{2})}(m\omega,t)$ so that the result is proved for all $\lambda\in \Z$.

Finally, for $m,n \neq 0$, the above result with $\omega,m$ replaced by $\frac{m\omega}{n},n$ implies that $\wp^{(j_{1},j_{2})}(m\omega,t)$ is algebraically related to $\wp^{(j_{1},j_{2})}(\frac{m}{n}\omega,t)$ over $\C\bigl(g_2^{(j_{1},j_{2})},g_3^{(j_{1},j_{2})}\bigr)$. Hence, since algebraic dependency is a transitive relation, for all $m,n \neq 0$, $\wp^{(j_{1},j_{2})}(\frac{m}{n}\omega,t)$    is algebraically related to $\wp^{(j_{1},j_{2})}(\omega,t)$ over $\C\bigl(g_2^{(j_{1},j_{2})},g_3^{(j_{1},j_{2})}\bigr)$. The proof is complete.
\end{proof}

\begin{lem}
\label{lem6}
The elliptic functions $\wp^{(k,k)}(\omega,t)$ and $\partial_{\omega}\wp^{(k,k)} (\omega,t)$ are algebraic over ${\cx}$. 
\end{lem}

\begin{proof}
Using a reasoning similar as in Lemma~\ref{lem2}, we obtain that  
\begin{equation*}
   \wp^{(k,k)}(k\omega ,t) = \frac{1}{k^2}\wp^{(1,1)}(\omega ,t).
\end{equation*}
By Lemma~\ref{lem5}, $\wp^{(k,k)}(k\omega,t)$ is algebraic over ${\cx}$. It then follows from Lemma~\ref{lem:Wal} that the function $\wp^{(k,k)}(\omega,t)$ is algebraic over ${\cx}$. 
 Finally, since $\partial_{\omega}\wp^{(k,k)}(\omega ,t)$ is algebraic over the field ${\C\bigl(g_2^{(k,k)},g_3^{(k,k)},\wp^{(k,k)}(\omega,t)\bigr)},$ the result for $\partial_{\omega}\wp^{(k,k)} (\omega,t)$ follows from Lemma~\ref{lem2}. 
\end{proof}

Proving the algebraicity of $ \wp^{(1,k)}$ happens to be more delicate.  We first need to show the algebraicity of values of $\wp^{(k,k)}$ at rational multiples of $\omega_1$ and $\omega_2$.  

\begin{lem}
\label{lem3}
For all $\lambda\in \Q\cap (0,k)$, the quantities $\wp^{(k,k)}(\lambda \omega_1(t),t)$ and $\partial_{\omega}\wp^{(k,k)}(\lambda \omega_1(t),t)$ are algebraic over $\C(t)$. A similar statement holds when $\omega_1$ is replaced by $\omega_2$. 
\end{lem}

\begin{proof}
The proof when $\omega_1$ is replaced by $\omega_2$ is similar, so we will focus on $\omega_1$.  
Since $\partial_{\omega}  \wp^{(k,k)}(k\omega_1 (t)/2,t)=0$, it follows from the differential equation satisfied by $\wp^{(k,k)}$ that  $\wp^{(k,k)}(k\omega_1(t)/2,t)$ is algebraic over $\C(g_2^{(k,k)},g_3^{(k,k)})$, and hence over $\mathbb{C}(t)$ by Lemma~\ref{lem2}. 

 Let $\lambda\in \Q\cap (0,k)$.  By Lemma~\ref{lem:Wal}  
 the function $x\mapsto \wp^{(k,k)}(x\lambda  \omega_1(t) ,t)$  is algebraic over 
 \begin{equation}
 \label{eq:ext}
     \C\left(g_2^{(k,k)},g_3^{(k,k)},{\wp^{(k,k)}(x k\omega_1 (t)/2 ,t)}\right).
 \end{equation}
 Hence, there exists some non-trivial polynomial $P_{x}$ with coefficients in \eqref{eq:ext} satisfying \begin{equation}
     \label{eq:pol_Px}
     P_{x}\bigl(\wp^{(k,k)}(x\lambda  \omega_1(t) ,t)\bigr)=0
 \end{equation} for all $x$.
 We multiply  $P_{x}$ by an appropriate power of $x-1$, so that at least one of its coefficients does not vanish at $x=1$ and none of them has a pole at $x=1$. At $x=1$, the relation \eqref{eq:pol_Px} reduces to a non-trivial polynomial of $\wp^{(k,k)}(\lambda  \omega_1(t) ,t)$ with coefficients in $\C\bigl(t,g_2^{(k,k)},g_3^{(k,k)},\wp^{(k,k)}(k\omega_1(t) /2,t)\bigr)$. Since $g_2^{(k,k)}$, $g_3^{(k,k)}$ and $\wp^{(k,k)}(k\omega_1(t) /2,t)$ are algebraic over $\C(t)$, this implies that $\wp^{(k,k)}(\lambda\omega_1(t) ,t)$ is also algebraic over $\C(t)$.    
This proves the lemma for $\wp^{(k,k)}(\lambda\omega_1(t) ,t)$. Finally, the result follows as $\partial_{\omega}\wp^{(k,k)}(\lambda \omega_1(t) ,t)$ is algebraic over the field 
\begin{equation*}
   {\C\bigl(g_2^{(k,k)},g_3^{(k,k)},\wp^{(k,k)}(\lambda \omega_1(t) ,t)\bigr)}.\qedhere
\end{equation*}
\end{proof}

\begin{lem}
\label{lem7}
The elliptic functions $\wp^{(1,k)}(\omega,t)$ and $\partial_{\omega}\wp^{(1,k)}(\omega,t)$ are algebraic over ${\cx}$. 
\end{lem}

\begin{proof}
Let us recall the classical identity
\begin{equation}\label{eq2}
\sum_{\ell=0}^{k-1} \wp^{(k,k)}(\omega + \ell\omega_1 (t) ,t)=\wp^{(1,k)}(\omega,t)+\sum_{\ell=1}^{k-1} \wp^{(k,k)}( \ell\omega_1 (t),t),
\end{equation}
which is obtained as follows. Both sides are $(\omega_1 (t) ,k\omega_2(t))$-periodic with only potential poles at $\omega_1(t)\Z+k\omega_2(t)\Z$.  The expansion at $\omega = 0$ of both sides is $\omega^{-2}+O(\omega)$, proving that the difference
\begin{equation*}
   d(\omega):=\sum_{\ell=0}^{k-1} \wp^{(k,k)} (\omega + \ell\omega_1 (t),t)-\wp^{(1,k)}(\omega,t)-\sum_{\ell=1}^{k-1} \wp^{(k,k)}( \ell\omega_1(t) ,t)
\end{equation*}
has no pole at $\omega=0$. 
Hence, $d(\omega)$ is an $(\omega_1 (t) ,k\omega_2(t))$-periodic function with no poles, so it is constant. Moreover, expanding one term further at $\omega = 0$ shows that $d(0)=0$, so $d(\omega)=0$ for all $\omega$.  This proves \eqref{eq2}. We note that \eqref{eq2} can alternatively be deduced directly from the definition of $\wp$.

We now use the addition formula \eqref{eq:addition_formula} of the Weierstrass function 
\begin{multline*}
\wp^{(k,k)}(\omega,t)+\wp^{(k,k)}(\ell\omega_1 (t),t)+\wp^{(k,k)}(\omega+\ell\omega_1 (t) ,t)\\
=\frac{1}{4}\left(\frac{\partial_{\omega}\wp^{(k,k)}(\omega,t)-\partial_{\omega}\wp^{(k,k)}(\ell\omega_1 (t) ,t)}{\wp^{(k,k)}(\omega,t)-\wp^{(k,k)}(\ell\omega_1 (t) ,t)}\right)^2, 
\end{multline*}   
which shows that the sum
$
\sum_{\ell=0}^{k-1} \wp^{(k,k)}(\omega + \ell\omega_1 (t) ,t)$ belongs to 
\begin{equation*}
   \Q\bigl(\wp^{(k,k)}(\omega,t),  \wp^{(k,k)}(\ell\omega_1 (t) ,t),\partial_{\omega}\wp^{(k,k)}(\omega,t),  \partial_{\omega}\wp^{(k,k)}(\ell\omega_1 (t) ,t)\bigr)_{0<\ell<k}.
\end{equation*}
By Lemmas~\ref{lem6} and \ref{lem3},  this sum  is algebraic over ${\cx}$.  Furthermore,  by Lemma~\ref{lem3} again,  $\sum_{\ell=1}^{k-1} \wp^{(k,k)}( \ell\omega_1  (t),t)$ is algebraic over ${\cx}$. Using \eqref{eq2}, this allows us to conclude that  $\wp^{(1,k)}(\omega,t)$  is algebraic over ${\cx}$. Finally, by Lemma~\ref{lem:deriv}, $\partial_{\omega} \wp^{(1,k)}$ is also algebraic over ${\cx}$.
\end{proof}

\begin{lem}
\label{lem8}
The invariants $g_2^{(1,k)}$ and $g_3^{(1,k)}$ are algebraic over $\mathbb C(t)$.  
\end{lem}

\begin{proof} 
Combining Lemmas~\ref{lem:deriv} and~\ref{lem7},  $\partial_{\omega}^2 \wp^{(1,k)}$ is algebraic over ${\cx}$.
Hence, the differential equation  $\partial_{\omega}^2 \wp^{(1,k)}=6(\wp^{(1,k)})^2-\frac{g_2^{(1,k)}}{2}$,  along with Lemma~\ref{lem7} again implies that $g_2^{(1,k)}$ is  algebraic over ${\cx}$. 
Finally, using the differential equation
\begin{equation*}
   (\partial_{\omega}\wp^{(1,k)}(\omega,t))^2= 4 \wp^{(1,k)}(\omega,t)^3-g_2^{(1,k)}\wp^{(1,k)}(\omega,t)-g_3^{(1,k)},
\end{equation*}
we deduce that $g_3^{(1,k)}$ is algebraic over ${\cx}$ too.   
\end{proof}

The proof of Theorem~\ref{thm3} is complete.

\subsection{A proof of Theorem~\ref{thm1}}

In this part, we check the assumptions~\ref{thm2:condition_elliptic},~\ref{thm2:condition_nicepoles},~\ref{thm2:condition_principalcoefficients} and~\ref{thm2:condition_existsnicea} of Theorem~\ref{thm2}, with $\mathbf{k}=\mathbf{R}=\overline{{\cx}}$ and $f(\omega,t)=\rx(\omega,t)$, where we recall that $r_x (\omega,t)$ is the meromorphic continuation of $K(x(\omega,t),0,t)Q(x(\omega,t),0,t)$, see \eqref{eq:liftings_GF}. We do it in two separate lemmas: while the first one (Lemma~\ref{lem1:1}) is specific to the algebraic case considered in this section, the second one (Lemma~\ref{lem1:2}) will be used and applied exactly in the same way in the D-finite case analyzed in Section~\ref{sec:DF_case}.

Before these two lemmas, we prove a general result about principal parts.
\begin{lem}\label{lem:alg_coeffs}
Let $\mathbf{k}=\overline{{\cx}}$ and  $\mathbf{k}_t=\overline{\mathbb C(t)}$ and recall the definition \eqref{eq:def_X_t} of $X_{t}$ in terms of $\mathbf{k}_t$. For any function $h(\omega,t)\in\mathbf{k}$, and any $b(t)\in X_{t}$, the function $h(\omega,t)$ expands as a series around $\omega=b(t)$, and all coefficients lie in $\mathbf{k}_{t}$. In particular the principal parts of $h(\omega,t)$ lie in $\mathbf{k}_{t}$.
\end{lem}

\begin{proof}
If $b(t)$ is not a pole of $h(\omega,t)$, then it suffices to prove that $\partial_{\omega}^{j}h(b(t),t)\in\mathbf{k}_t$ for all $j$. Since $b(t)\in X_t$, then $\wp(b(t),t)\in  \mathbf{k}_t\cup \{\infty\}$ and with Proposition~\ref{prop:uniformization} we deduce that $x(b(t),t)\in \mathbf{k}_t$, so $h(b(t),t)\in \mathbf{k}_t$. With  Lemma~\ref{lem:deriv}, for all $j$, $\partial_{\omega}^{j}h(\omega,t)$ is algebraic over ${\cx}$.
Hence, we  deduce that $\partial_{\omega}^{j}h(b(t),t)\in \overline{\mathbf{k}_t}=\mathbf{k}_t$ for all $j$.

If now $b(t)$ is a pole of $h(\omega,t)$, we can use the same reasoning to show that the lemma holds for $1/h(\omega,t)$, which implies that it also holds for $h(\omega,t)$. This proves our claim, which in particular proves that the coefficients of the principal parts of $h(\omega,t)$ lie in $\mathbf{k}_t$.
\end{proof}

\begin{lem}
\label{lem1:1}
Let $\mathbf{k}=\overline{{\cx}}$ and  $\mathbf{k}_t=\overline{\mathbb C(t)}$. The following holds:
\begin{enumerate}[label={\rm(\roman*)},ref={\rm(\roman*)}]
   \item For all $t\in (0,1)$, $\omega\mapsto r_{x}(\omega,t)\in \C\bigl(\wp^{(1,k)},\partial_{\omega}\wp^{(1,k)}\bigr)$. \label{lem1:condition1}
\end{enumerate}
\end{lem}

\begin{lem}
\label{lem1:2}
Under the same notation of Lemma~\ref{lem1:1}, the following holds:
\begin{enumerate}[label={\rm(\roman*)},ref={\rm(\roman*)}]
\setcounter{enumi}{1}
\item The  poles $b_j(t)$ of $\omega\mapsto r_{x}(\omega,t)$ belong to $X_t$.\label{lem1:condition2}
\item The coefficients of the principal parts of $\omega\mapsto r_{x}(\omega,t)$ are in $\mathbf{k}_t$. \label{lem1:condition3}
\item There exists $a(t)\in X_t$ such that $r_{x}(a(t),t)\in \mathbf{k}_t$. \label{lem1:condition4}
\end{enumerate}
\end{lem}

To simplify our notation, we will sometimes remove the $t$-dependency and for instance write $x(\omega)$ instead of $x(\omega,t)$. 

\begin{proof}[Proof of Lemma~\ref{lem1:1}]
Condition~\ref{lem1:condition1} is equivalent to the statement that $r_{x}(\omega)$ has $\omega_{1}$ and $k\omega_{2}$ as periods. Indeed, the fact that $\omega_{1}$ is a period follows from the definition, see \eqref{eq:omega_1_per_rx}, while the fact that $k\omega_{2}$ is a period follows from \eqref{eqrx_transform_finite_group}, since we are in the case where the orbit-sum $\mathcal{O}_{x}$ is zero.
 \end{proof}
 
\begin{proof}[Proof of Lemma~\ref{lem1:2}]
We start with proving Condition~\ref{lem1:condition2}. We will heavily rely on a result derived in \cite[Sec.~2.6]{dreyfus2021differential}, 
asserting that all poles $b_{j}$ of $\omega \mapsto r_x (\omega)$ are of the form $b_j=b+\ell \omega_3$, for some $\ell\in \Z$  and some pole $b$ of either $x(\omega)$ or $b_x(\omega)=y(-\omega)(x(\omega)-x(\omega+\omega_3))$, see \eqref{eq:def_bx_by}. This follows from the meromorphic continuation procedure \eqref{eq:omega_3_per_rx}. We may reformulate this result by saying that the poles of $r_x$ are to be found among the points $b_j=b+\ell \omega_3$, for some $\ell\in\Z$ and $b$ pole of $x(\omega)$ or $y(-\omega)$.

Given the definition \eqref{eq:def_X_t} of $X_t$, Condition~\ref{lem1:condition2} will follow from proving that for any $b$ pole of $x(\omega)$ or $y(-\omega)$ and any $\ell\in\mathbb Z$, $\wp^{(1,k)}(b+\ell \omega_3)$ is either infinite or algebraic over $\mathbb C(t)$. The strategy of proof is as follows: we shall first prove the statement for the Weierstrass function $\wp^{(1,1)}$, then for $\wp^{(1,k)}$.

If $b$ is a pole of $x(\omega)$, then using the uniformization (Proposition~\ref{prop:uniformization}), we reach the conclusion that $\wp^{(1,1)}(b)$ is algebraic over $\mathbb C(t)$. If now $b$ is a pole of $y(-\omega)$, it follows from the equation $K(x(\omega),y(-\omega))=0$ that either $b$ is a pole of $x(\omega)$ or $x(b)$ is algebraic over $\mathbb{C}(t)$. The uniformization formula implies that $\wp^{(1,1)}(b)$ is algebraic or infinite. Using now the addition formula \eqref{eq:addition_formula} for $\wp^{(1,1)}$, we deduce that $\wp^{(1,1)}(b+\ell \omega_3)$ is algebraic as well, or infinite. Indeed, recall that in the finite group case $\omega_3$ is a rational multiple of $\omega_2$, hence Lemma~\ref{lem3} yields that, if finite, $\wp^{(1,1)}(\ell \omega_3)$ and $\partial_\omega\wp^{(1,1)}(\ell \omega_3)$ are algebraic functions of $t$. Using finally Lemma~\ref{lem7}, we immediately conclude that $\wp^{(1,k)}(b+\ell \omega_3)$ is either infinite or algebraic over $\mathbb C(t)$.

We now prove Condition~\ref{lem1:condition3}. We start by analysing $x(\omega)y(\omega)$ around a pole $b(t)$ of $r_{x}(\omega,t)$, as this will relate to a principal part of $r_{x}$ due to \eqref{eq:lifting_func_eq}. By Lemma \ref{lem:alg_coeffs}, the coefficients of the principal parts of $x(\omega)y(\omega)$ (at poles $b(t)$) lie in $\mathbf{k}_t$. We now consider the poles of $r_x$ belonging to the domain $\mathcal O$ (see Lemma~\ref{lem:existence_properties_O} for a proper definition of $\mathcal{O}$). In $\mathcal O$, the function $\rx$ is defined in the following way. Since the series  $K(x,0)Q(x,0)$ converges absolutely when $\vert x\vert<1$ and $0<t<1$, and since $r_x (\omega)$ is the meromorphic continuation of $K(x(\omega),0)Q(x(\omega),0)$, the function $r_x (\omega)$ has no pole when $\vert x(\omega)\vert<1$ and $\omega\in \mathcal{O}$. In the region $\vert y(\omega)\vert<1$ and $\omega\in \mathcal{O}$, the functions $r_x (\omega)$ and $-x(\omega)y(\omega)$ have the same poles and the same principal parts at these poles, as 
$\rx (\omega)+\ry(\omega)+x(\omega)y(\omega)-K(0,0)Q(0,0)=0$ (see \eqref{eq:lifting_func_eq}).
As a partial conclusion, the coefficients of the principal parts of $\omega \mapsto r_x (\omega)$ at any pole inside $\mathcal{O}$ are algebraic.

Let us now use the fact that the union of all translated domains $\mathcal O + \ell\omega_3$, $\ell\in\mathbb Z$ covers the full complex plane, as stated in Lemma~\ref{lem:existence_properties_O}. Moreover, the following equation holds: $\rx (\omega+\omega_3)=\rx (\omega)+b_x (\omega)$, see \eqref{eq:omega_3_per_rx}. Accordingly, the poles of $r_x$ come from exactly two sources: first, all points which are poles of $b_x$ translated by $\ell\omega_3$ are potential poles of $r_x$; second, the $\ell\omega_3$-translations of the primordial poles of $r_x$ in the region $\mathcal O$ represent potential poles of $r_x$ as well. The coefficients of the principal parts being unaffected by the translations described by the dynamic \eqref{eq:omega_3_per_rx}, the proof is complete.

 Finally, we show that Condition~\ref{lem1:condition4} holds.  Similarly to  \cite[Lem.~3.11]{dreyfus2021differential}, consider a root 
 $\omega_0=\omega_0(t)$ of $y(\omega_0 )=0$.
 Let us show that we may take the value $a(t)=\omega_0 $.  Since $\wp^{(1,k)}$ is algebraic over ${\cx}$,  see Lemma~\ref{lem7}, we deduce  that either $\omega_0$ is a pole of $ \wp^{(1,k)}$, or 
  $\wp^{(1,k)}(\omega_0)$ is algebraic over $\mathbb{C}(t)$.  So it remains to consider the case where  $\wp^{(1,k)}(\omega_0)$ is algebraic over $\mathbb{C}(t)$.

In the proof of \cite[Lem.~3.11]{dreyfus2021differential}, it is shown that $\rx (\omega_0)+x(\omega_0)y(\omega_0)=0$. Now, by Proposition~\ref{prop:uniformization} and the algebraicity of $\wp^{(1,1)}(\omega_0)$ over $\mathbb{C}(t)$, we deduce that $x(\omega_0)y(\omega_0)$ is algebraic over $\mathbb{C}(t)$. Then $\rx (\omega_0)$ is algebraic over $\mathbb{C}(t)$. This concludes the proof.  
\end{proof}

We are now ready to prove the main result of the section.  

\begin{proof}[Proof of Theorem~\ref{thm1}]
Recall from the very beginning of Section~\ref{secalg} that we have already proved that the algebraicity property of $Q(x,y,t)$ implies the zero orbit-sum condition. It remains to prove the converse statement. To do so, assume that
 $\mathcal{O}_x(\omega,t)=\mathcal{O}_y(\omega,t)=0$ for all $t\in (0,1)$.  By \eqref{eq:funcequ}, it suffices to show that $K(x,0,t)Q(x,0,t)$ and $K(0,y,t)Q(0,y,t)$ are algebraic over $\C(x,t)$ and $\C(y,t)$ respectively.  
Let us consider $K(x,0,t)Q(x,0,t)$, the proof for $K(0,y,t)Q(0,y,t)$ is similar.  We have seen that  $K(x,0,t)Q(x,0,t)$  admits an analytic continuation $r_x (\omega,t)$ that is $(\omega(t), k\omega_2 (t))$-periodic.  
 By Theorem~\ref{thm3}, Lemma~\ref{lem1:1} and Lemma~\ref{lem1:2},  we may apply Theorem~\ref{thm2} with $\mathbf{k}=\mathbf{R}=\overline{{\cx}}$ and $f(\omega,t)=\rx(\omega,t)$, and deduce that   $r_x (\omega,t)$ is algebraic over $\overline{{\cx}}$. Then  $K(x,0,t)Q(x,0,t)$ is algebraic over $\C(x,t)$. 
\end{proof}

\section{D-finite case: a detour via Weierstrass zeta functions}
\label{sec:DF_case}

The main result in this section is Theorem~\ref{thm:main_section_DF}, which says the following: 
Assume that the group of the walk is finite. Then $Q(x,y,t)$ satisfies a (non-trivial)\ linear differential equation coefficients in $\C(x,y,t)$ in each of its variables.

\subsection{Strategy of the proof}

The first point is that it is sufficient to prove Theorem~\ref{thm:main_section_DF} for the specialization $Q(x,0,t)$, by similar arguments as in Section~\ref{secalg} (using standard manipulations on the functional equation~\eqref{eq:funcequ}). To that purpose, we shall prove a factorization of the generating function $Q(x,0,t)$ as follows:
\begin{equation}
\label{eq4x}
   K(x,0,t)Q(x,0,t) = F_1(x,t)F_2(x,t)+F_3(x,t),
\end{equation}
and we will successively show that $F_1$, $F_2$ and $F_3$ are $D$-finite functions in both variables $x$ and $t$ (concluding thanks to the ring structure of D-finite functions).

To prove \eqref{eq4x}, our starting point is the lifting $\rx(\omega,t)$ of $Q(x,0,t)$ as introduced in \eqref{eq:liftings_GF}, which we will prove to satisfy the following decomposition:
\begin{equation}
\label{eq4}
   \rx (\omega,t)= \mathcal{O}_x (\omega,t)\phi (\omega,t)+\psi (\omega,t),
\end{equation}
where
\begin{itemize}
   \item $\mathcal{O}_x (\omega,t)$ is the orbit-sum \eqref{eq:orbit_sums};
   \item $\phi (\omega,t)=\frac{\omega_{1}(t)}{2\mathbf{i}\pi}\zeta (\omega,t)-\frac{\omega}{\mathbf{i}\pi}\zeta (\tfrac{\omega_{1}(t)}{2},t)$, with $\zeta$ the negative of an anti-derivative of the Weierstrass function with periods $(\omega_{1}(t),k\omega_{2}(t))$, i.e., $\partial_\omega \zeta(\omega,t) = - \wp^{(1,k)}(\omega,t)$;
    \item $\psi (\omega,t)$ is a certain $(\omega_{1}(t),k\omega_{2}(t))$-periodic function.
\end{itemize}
This decomposition is obtained in \cite[Sec.~4]{dreyfus2019differential} (see also \cite[Sec.~9]{new} for unweighted quadrant walks), based on ideas of \cite{FIM17}.

In the zero orbit-sum case (Section~\ref{secalg}), the function $r_x(\omega,t)=\psi (\omega,t)$ is $(\omega_1(t),k\omega_2(t))$-elliptic by \eqref{eq4}; see also \eqref{eqrx_transform_finite_group} and recall that we have $\omega_3 /\omega_2=k/\ell$, see \eqref{eq:ratio_periods}. This strong periodicity property is no longer satisfied  in the non-zero orbit-sum case. 

Equation~\eqref{eq4} immediately entails the identity~\eqref{eq4x},
where for instance (with obvious notation) $F_1(x(\omega,t),t)=\mathcal{O}_x (\omega,t)$, and thus $F_1$ is a pullback of the orbit-sum.

It is crucial to remark that the role and the complexity of the three functions appearing in \eqref{eq4x} is very different. Indeed, the function $F_1$ is a simple, explicit function, and we will easily prove in Lemma~\ref{lem16} that it is algebraic and therefore D-finite. The function $F_2$ is also explicit but contained some more difficulties, as it implies the (intrinsically transcendental)\ zeta Weierstrass function; it will be proved to be D-finite in Lemma~\ref{lem15} (as a function of $x(\omega,t)$ and $t$). On the other hand, we only have an implicit control on the function $F_3$, through the poles of the elliptic function $\psi$ in~\eqref{eq4}. Using precise informations on the poles of various functions as well as Theorem~\ref{thm2}, we shall obtain that $F_3$ is D-finite up to some additive constant, see Lemma~\ref{lem17}.

\subsection{Preliminary results on D-finiteness}

We start with two classical results, which will be repeatedly used in Section~\ref{sec:DF_case}.   

\begin{lem}
\label{lem13}
Let $(k,\partial)$ be a differential field and let $(R,\partial)$ be a differential ring extension of $(k,\partial)$. Then $f\in R$ satisfies a linear $\partial$-equation with coefficients in $k$ if and only if the dimension of the $k$-vector space $\mathrm{Vect}_{k} (\partial^{n}f)_{n\in \N}$ is finite. 
\end{lem}

\begin{lem}
\label{lem14}
Let $(k,\partial)$ be a differential field and let $(\overline{k},\partial)$ be its algebraic closure. Let $(R,\partial)$ be a differential ring extension of $(k,\partial)$.  If $f\in R$  satisfies a linear $\partial$-equation with coefficients in $\overline{k}$, then it satisfies a linear $\partial$-equation with coefficients in $k$.
\end{lem}

\begin{proof}
Lemma~\ref{lem13} admits an elementary proof. Lemma~\ref{lem14} is classical, but let us give a short proof. As for Lemma~\ref{lem14}, assume that $f\in R$  satisfies a linear $\partial$-equation with coefficients in $\overline{k}$. Let $L$ be a finite field extension of  $k$ that contains the coefficients of the differential equation. By Lemma~\ref{lem13}, the dimension of the vector space $\mathrm{Vect}_{L} (\partial^{n}f)_{n\in \N}$ is finite.  Since $L\vert k$ is finite too, the dimension  $\mathrm{Vect}_{k} (\partial^{n}f)_{n\in \N}$ is finite and by Lemma~\ref{lem13}, $f$ satisfies a linear $\partial$-equation in coefficients in $k$.  
\end{proof}

\subsection{D-finiteness and pullbacks}

Recall that the generating function $K(x,0,t)Q(x,0,t)$ admits a meromorphic lifting $r_x (\omega,t)$ on the $\omega$-complex plane, see~\eqref{eq:liftings_GF}. Let us first see how the D-finiteness of an $(x,t)$-function is related to the D-finiteness of the associated lifted $(\omega,t)$-function.

\begin{prop}
\label{prop:DF_diff_eval}
Let $f(\omega,t)$ be a function such that for all $t\in (0,1)$, $\omega\mapsto f(\omega,t)$ is meromorphic on $\C$. In addition, let $a(t)\in \overline{\mathbb{C}(t)}$ and $\omega(t)$ be such that $x(\omega (t),t)=a(t)$. Assume that $\partial_{\omega} f\in \overline{{\cx}}$ and that $t \mapsto f(\omega(t),t)$ is  D-finite over $\mathbb{C}(t)$. Then any differentiable function $F(x ,t)$ satisfying $f(\omega,t)=F(x(\omega,t),t)$ on some neighbourhood is D-finite in its two variables over $\mathbb{C}(x,t)$.
\end{prop}

\begin{proof}
Differentiating the identity relating $f$ and $F$, one finds 
\begin{equation}
\label{eq:diff_lift}
   \partial_{\omega} f (\omega,t)= \partial_{\omega} x(\omega,t) \partial_{x}F(x (\omega,t),t).
\end{equation}
By Lemma~\ref{lem:deriv} (resp.\ by assumption), the function $\partial_{\omega} x(\omega,t)$ (resp.\ $\partial_{\omega} f (\omega,t)$) belongs to $\overline{{\cx}}$.  Equation~\eqref{eq:diff_lift} allows us to obtain that $\partial_x F(x,t)$ is algebraic over $\mathbb{C}(x,t)$. Since any algebraic function is D-finite, see \cite[Prop.~2.3]{lipshitz1989d}, we reach the conclusion that $F(x,t)$ satisfies a linear differential equation in its first variable. 

We now look at the $t$-variable. Fix a non-critical point $(x_{c},t_{c})$ within the neighbourhood where $F(x,t)$ is defined and consider the expansion
\[\partial_x F(x,t)=\sum_{j=0}^{\infty}\sum_{k=0}^{\infty}(x-x_{c})^{k}(t-t_{c})^{j}a_{i,j}\]
around this point. Following the discussion preceding \cite[Thm~2.7]{lipshitz1989d}, we consider the function
\begin{equation*}
   \widetilde{F}(x,t):=\bigl((x-x_{c})\partial_x F(x,t)\bigr)\star\left(-\frac{\log(1-(x-x_{c}))}{1-(t-t_{c})}\right)=\sum_{j=0}^{\infty}\sum_{k=1}^{\infty}\frac{(x-x_{c})^{k+1}}{k+1}(t-t_{c})^{j}a_{i,j},
\end{equation*}
where $\star$ denotes the Hadamard product with respect to expansions around $(x_{c},t_{c})$. Since the Hadamard product of D-finite functions is D-finite, see  \cite[Thm~2.7]{lipshitz1989d}, the function $\widetilde{F}(x,t)$ is D-finite (in both variables). Moreover, due to the expansion, $\partial_x \widetilde{F}(x,t)=\partial_x F(x,t)$, that is $F(x,t)=\widetilde{F}(x,t)+A(t)$ for some function $A(t)$ not depending on $x$. So it suffices to show that $A(t)$ is D-finite.

We now consider a function $F_{1}(x,t)$ defined in a connected region containing $x=a(t)$ and satisfying $F_{1}(x(\omega,t),t)=f(\omega,t)$ for $\omega$ in a connected region containing $\omega=\omega(t)$. By the meromorphicity of $f$, the function $F_{1}(x,t)-A(t)$ and $\widetilde{F}(x,t)=F(x,t)-A(t)$ are sheets of the same multi-valued function. This implies that $F_{1}(x,t)-A(t)$ is D-finite in both variables as it satisfies the same linear differential equations as $\widetilde{F}(x,t)$. Since D-finite functions of algebraic functions are necessarily D-finite, see \cite[Prop~2.3]{lipshitz1989d}, this implies that $F_{1}(a(t),t)-A(t)$ is D-finite in $t$. Moreover, $F_{1}(a(t),t)=f(\omega(t),t)$, which is D-finite in $t$ by assumption, so $A(t)$ is also D-finite. Hence $F(x,t)=\widetilde{F}(x,t)+A(t)$ is D-finite as claimed.
\end{proof}

\subsection{D-finiteness of the functions $F_1$, $F_2$ and $F_3$}

In this part, we show crucial preliminary results to the proof of Theorem~\ref{thm:main_section_DF}.

\begin{lem}\label{lem16}
The function $\mathcal{O}_x$ defined in \eqref{eq:orbit_sum_Ox} belongs to $ \overline{{\cx}}$. 
\end{lem}

\begin{proof}
Clearly with \eqref{eq:orbit_sum_Ox} we have $\mathcal{O}_x \in \C(x(\omega,t),y(\omega,t),t)$. We conclude with the fact that $y(\omega,t)\in \overline{{\cx}}$, due to the polynomial equation $K(x(\omega,t),y(\omega,t),t)=0$.
\end{proof}

Let us now consider the pullback of the function $\phi$ appearing in \eqref{eq4}. Write $\phi(\omega,t)=F_2(x(\omega,t),t)$, with
\begin{equation}
\label{eq:def_F2_G_H}
   F_2(x,t)=\frac{\omega_{1}(t)}{2\mathbf{i}\pi}G(x,t)-\frac{H(x,t)}{\mathbf{i}\pi}\zeta \bigl(\tfrac{\omega_{1}(t)}{2},t\bigr).
\end{equation}
We have the following result:
\begin{lem}
\label{lem15}
The function $F_2(x,t)$ satisfies a linear differential equation with coefficients in $\mathbb{C}(x,t)$ in each of its variables.  
\end{lem}

Before embarking into the proof of Lemma~\ref{lem15}, we show the following properties of the Weierstrass zeta function:
\begin{lem}
\label{lem:zeta_DF}
Let $\zeta$ and $\wp$ denote, respectively, the Weierstrass $\zeta$-function and $\wp$-function with periods $(\omega_{1}(t),k\omega_{2}(t))$. We have $\partial_{\omega} \zeta\in \overline{{\cx}}$ and for $a(t)\in X_{t}$, see \eqref{eq:def_X_t}, the functions $a(t)$ and $\zeta(a(t),t)$ are both D-finite with respect to $t$.
\end{lem}

\begin{proof}
We have $\partial_{\omega}\zeta (\omega,t)=-\wp(\omega,t)$, which belongs to $ \overline{{\cx}}$ by Theorem~\ref{thm3}. We now consider the $t$-derivation.  Let us see $\zeta$ and $\wp$ as functions of the invariants 
\begin{equation*}
   g_2:=g_2^{(1,k)} \quad \text{and} \quad g_3:=g_3^{(1,k)}
\end{equation*}
corresponding to the elliptic curve with periods $(\omega_1,k\omega_2)$. By \cite[(18.6.19-22)]{abramowitz1988handbook}, one has
\begin{align*}
   (g_2^3 -27 g_3^2)\partial_{g_2}\wp &=\wp' \left(-\frac{9}{2}g_3\zeta+\frac{1}{4}g_2^2\omega\right)-9g_3\wp^2+\frac{g_{2}^{2}}{2}\wp+\frac{3}{2}g_{2}g_{3},\\
   (g_2^3 -27 g_3^2)\partial_{g_3}\wp &=\wp' \left(3g_2\zeta-\frac{9}{2}g_3\omega\right)+6g_2\wp^2-9g_{3}\wp-g_{2}^2,\\
   (g_2^3 -27 g_3^2)\partial_{g_2}\zeta &=\frac{1}{2}\zeta \left(9g_3\wp+\frac{1}{2}g_2^2\right)-\frac{1}{2}\omega \left(\frac{1}{2}g_2\wp+\frac{3}{4}g_3 \right)+\frac{9}{4}g_3\partial_{\omega}  \wp,\\
   (g_2^3 -27 g_3^2)\partial_{g_3}\zeta &=-3\zeta \left(g_2\wp+\frac{3}{2}g_3\right)+\frac{1}{2}\omega \left(9g_3\wp+\frac{1}{2}g_2^2 \right)+\frac{3}{2}g_2\partial_{\omega}  \wp,
\end{align*}
which may be each rewritten as an equation of the form,  
\begin{equation*}
   \partial_{g_j}y =a_j\zeta +b_j\omega +c_{j},\quad j=2,3,
\end{equation*}
where $a_j$, $b_{j}$ and $c_{j}$ are algebraic over ${\cx}$ and $y=\zeta$ or $\wp$. Substituting the equations with $y=\zeta$, into $\partial_{t}\zeta=(\partial_{t}g_2) \partial_{g_2}\zeta+(\partial_{t}g_3) \partial_{g_3}\zeta$, yields 
\begin{equation}
\label{eq:diff_zeta_t}
   \partial_{t}\zeta =(\partial_{t}g_2 )a_2\zeta +(\partial_{t}g_2 )b_2  \omega +(\partial_{t}g_2 )c_2+(\partial_{t}g_3 )a_3\zeta +(\partial_{t}g_3)b_3 \omega +(\partial_{t}g_3 )c_3.
\end{equation}
Note that  $\partial_{t}g_j $ is algebraic in $t$, as the derivative of an algebraic function (Lemma~\ref{lem4}), so we have $\partial_{t}\zeta(\omega,t)\in \overline{\cx}+\overline{\cx}\zeta(\omega,t)+\overline{\cx}\omega$. By the same analysis, $\partial_{t}\wp(\omega,t)\in \overline{\cx}+\overline{\cx}\zeta(\omega,t)+\overline{\cx}\omega$.

Now assume $a(t)\in X_{t}$, and introduce \begin{equation*}
    \Lambda_{t}=\mathbf{k}_t+\mathbf{k}_t\zeta(a(t),t)+\mathbf{k}_t a(t),
\end{equation*}
where $\mathbf{k}_t$ denotes the field of algebraic functions in $t$ (see Section~\ref{seccrit}). Setting $\omega=a(t)$ above yields
\begin{equation*}
   \partial_{t}\zeta(a(t),t),\partial_{t}\wp(a(t),t)\in \Lambda_{t}.
\end{equation*}
We will now show that $\Lambda_{t}$ is closed under differentiation by $t$. Note that $\mathbf{k}_t$ is closed under differentiation by $t$ and that $\Lambda_{t}$ is closed under multiplication by elements of $\mathbf{k}_t$, so it suffices to show that $a'(t)\in\Lambda_{t}$ and $\frac{d}{dt}\zeta(a(t),t)\in\Lambda_{t}$. First, recall that by the definition of $X_{t}$, along with Lemma \ref{lem7}, we have $\wp(a(t),t)\in\mathbf{k}_t$, so its derivative with respect to $t$ also lies in $\mathbf{k}_t$, that is
\begin{equation*}
   \partial_{\omega}\wp(a(t),t)a'(t)+\partial_{t}\wp(a(t),t)\in\mathbf{k}_t.
\end{equation*}
Combining this with $\partial_{t}\wp(a(t),t)\in\mathbf{k}_t \subset \Lambda_{t}$ and $\partial_{\omega}\wp(a(t),t)\in\mathbf{k}_t$ (from Lemma \ref{lem7}) yields $a'(t)\in\Lambda_{t}.$ Finally,
\begin{equation*}
   \frac{d}{dt}\zeta(a(t),t)=-\wp(a(t),t)a'(t)+\partial_{t}\zeta(a(t),t)\in\Lambda_{t}.
\end{equation*}
Hence $\Lambda_{t}$ is closed under differentiation with respect to $t$. Since $\Lambda_{t}$ has dimension (at most) $3$ as a vector space over $\mathbf{k}_t$, this implies that any element $f(t)\in\Lambda_{t}$ satisfies a non-trivial linear differential equation of order at most $3$ with coefficients in $\mathbf{k}_t$. By Lemma \ref{lem14}, this implies that $f(t)$ is D-finite. In particular, $a(t)$ and $\zeta(a(t),t)$ are both D-finite in $t$.
\end{proof}

\begin{proof}[Proof of Lemma~\ref{lem15}]
First, we observe that it is enough to prove that each of $G(x,t)$ and $H(x,t)$ in \eqref{eq:def_F2_G_H} have the above-mentioned property. Indeed, $\omega_{1}(t)$ is D-finite by \cite[Lem.~6.10]{BeBMRa-21} and so is $\zeta (\frac{\omega_{1}(t)}{2},t)$ by Lemma~\ref{lem:zeta_DF}; recall that D-finite functions are stable by addition and multiplication. 

In order to prove that $G(x,t)$ satisfies a linear differential equation with coefficients in $\mathbb{C}(x,t)$ in each of its variables, we shall apply Proposition~\ref{prop:DF_diff_eval} with $a(t) = x(\frac{\omega_1 (t)}{2},t)$, which is algebraic by Lemma~\ref{lem3} and Proposition \ref{prop:uniformization}. The result then directly follows from Lemma~\ref{lem:zeta_DF}.

We move to the function $H(x,t)$, applying once again Proposition~\ref{prop:DF_diff_eval}. The function $\partial_{\omega} H(x(\omega,t),t)=\partial_{\omega} \omega=1$ is obviously in $\overline{{\cx}}$. Moreover, choosing $a(t)=x(0,t)$, which by Proposition~\ref{prop:uniformization} is algebraic, we find $H(a(t),t)=0$. The proof is complete.
\end{proof}

From what precedes, $F_1(x,t)F_2 (x,t)$ is D-finite, hence $Q(x,0,t)$ is D-finite if and only if $K(x,0,t)Q(x,0,t)-F_1(x,t)F_2 (x,t)$ is D-finite. So by \eqref{eq4x} and \eqref{eq4}, it suffices to show that the pullback of $\psi$, namely $F_3$, is D-finite. This is exactly what the following result aims to achieve, up to an additive constant.

\begin{lem}
\label{lem17}
Let $F_3(x,t)$ be the pullback of $\psi$ appearing in \eqref{eq4} and let $x_0\in \mathbb{P}_{1}(\C)$ such that $K(x_0,0,t)=0$ (noting that $K(x_0,0,t)$ does not depend on $t$). Then $F_3(x,t)-F_3(x_0,t)$ is D-finite.
\end{lem}

As in the algebraic case, we will prove the above statement using Theorem~\ref{thm2}; we thus need to have a control on the poles and principal parts of $\psi$. Lemma~\ref{lem1:2} already contains precise properties of the poles and principal parts of $\rx$, and we now move to the function $\mathcal{O}_x \phi $. To that purpose, we first consider $\mathcal{O}_x$ and $\phi$ separately. Let us recall that the set $X_t$ has been defined above Theorem~\ref{thm2}, see \eqref{eq:def_X_t}. 
\begin{lem}
\label{lem18}
Let $f = \mathcal{O}_x$ or  $f=\phi $.
\begin{enumerate}
    \item\label{it:1:lem18}For all $t\in (0,1)$, the  poles $b_j(t)$ of $\omega\mapsto f(\omega,t)$ belong to $X_t$.
    \item\label{it:2:lem18}The coefficients of the expansion of $\omega\mapsto f(\omega,t)$ around any pole $b(t)$ of $\mathcal{O}_x$ or  $\phi $ are D-finite.
\end{enumerate}
\end{lem}
\begin{proof}
The poles of $\phi$ form the lattice $\omega_1 (t)\Z+ k\omega_2 (t)\Z$, while the poles of $\mathcal{O}_x$ are of the form $b+\ell \omega_3$, with $b$ a pole of $b_x(\omega,t)$. As we can see in the proof of Lemma~\ref{lem1:2},  $b+\ell \omega_3\in X_t$.  In both cases, this shows the first point \ref{it:1:lem18}.

For the second point, $\mathcal{O}_x\in\cx$, so the result for $f = \mathcal{O}_x$ follows from Lemma \ref{lem:alg_coeffs}.
Consider now $\phi$. By Lemma \ref{lem:zeta_DF}, each pole $b(t)$ is D-finite, so by the definition of $\phi$, it suffices to prove the same result for $\zeta(\omega,t)$. Note that $\partial_{\omega}\zeta(\omega,t)=-\wp^{(k,1)}(\omega,t)$. By Lemma~\ref{lem:alg_coeffs}, the coefficients of the expansion of $\wp^{(k,1)}(\omega,t)$ around $\omega=b(t)$ are algebraic, so the same holds for the non-constant coefficients of $\zeta(\omega,t)$. For the constant coefficients, if $b(t)$ is not a pole of $\zeta(\omega,t)$ then $\zeta(b(t),t)$ is D-finite by Lemma \ref{lem:zeta_DF}. If $b(t)$ is a pole of $\zeta(\omega,t)$, then $b(t)=m\omega_1 (t)+ kn\omega_2 (t)$ for some integers $m,n$ and the constant coefficient is $2m\zeta(\omega_{1}(t)/2,t)+2kn\zeta(k\omega_{t}(t)/2,t)$, which is D-finite.
\end{proof}

\begin{lem}\label{lem20}
The coefficients of the principal parts  of $\omega\mapsto \mathcal{O}_x(\omega,t)\phi(\omega,t)$ are D-finite.
\end{lem}

\begin{proof}
This follows immediately from Lemma~\ref{lem18} \ref{it:2:lem18}.
\end{proof}

\begin{proof}[Proof of Lemma~\ref{lem17}]
The poles of $\psi (\omega,t)=\rx (\omega,t)-\mathcal{O}_x (\omega,t)\phi (\omega,t)$ belong to $X_t$ by Lemma~\ref{lem18}. The principal parts of $\rx (\omega,t)$ are algebraic (and in particular D-finite). By Lemma~\ref{lem20}, the same holds for  the principal parts of $\omega\mapsto \mathcal{O}_x(\omega,t)\phi(\omega,t)$, and hence the principal parts of  $\psi$ are D-finite and its poles are algebraic.    
 
Consider now a complex number $\omega_0$ with $y(\omega_0 ,t)=0$ and  $x(\omega_0 ,t)=x_0$.  Then $\psi(\omega,t)-\psi(\omega_0,t)$ vanishes at $\omega_0$.  
Following the proof of Theorem~\ref{thm2}, the function $\psi$ admits a decomposition 
\begin{equation*}
   \psi(\omega,t) =\psi(\omega_0,t)+\sum_{i}f_i (t)\psi_i (\omega,t),
\end{equation*}
where $\psi_i\in \overline{{\cx}}$ and $f_i$ are D-finite.  So $F_3(x,t)-F_3(x_0,t)$  admits a decomposition into sum and product of D-finite elements, so is D-finite.  
\end{proof}

\subsection{Conclusion: proof of Theorem~\ref{thm:main_section_DF}}

\begin{proof}
Let us write 
\begin{equation*}
   \rx (\omega,t)-\psi(\omega_0,t)=\psi (\omega,t)-\psi(\omega_0,t)+ \mathcal{O}_x (\omega,t)\phi (\omega,t).
\end{equation*}
From what precedes $K(x,0,t)Q(x,0,t)-\psi(\omega_0,t)=K(x,0,t)Q(x,0,t)-F_3(x_0,t)$ is D-finite.  Then $K(x_0,0,t)Q(x_0,0,t)-F_3(x_0,t)$ is D-finite. We have seen in Lemma~\ref{lem1:2} that $\rx (\omega_0,t)=K(x_0,0,t)Q(x_0,0,t)$ is algebraic (and therefore D-finite). Then $F_3(x_0,t)$ is D-finite. Since both $F_3(x_0,t)$ and $K(x,0,t)Q(x,0,t)-F_3(x_0,t)$ are D-finite, we deduce  the D-finiteness of $K(x,0,t)Q(x,0,t)$. Hence  $Q(x,0,t)$ is D-finite.  Similarly, the D-finiteness of  $Q(0,y,t)$ follows.   We conclude with the functional equation \eqref{eq:funcequ}. 
\end{proof}

\section{Infinite group case}
\label{sec:infinite_group_case}

It remains to treat the infinite group case, and to provide the proof of Theorem~\ref{thm4}. 
\begin{proof}
Recall from \cite[Lem.~8.16]{EP-22} that some $\varepsilon>0$  necessarily exists such that for $t\in(0,\varepsilon)$, $Q(x,0,t)\notin \C(x)$ and $Q(0,y,t)\notin \C(y)$. More generally, Theorem 8.7 in that article shows that $Q(x,y,t)$ is D-finite in $x$ (or $y$) for all fixed $t\in(0,1)$ if and only if the group is finite. 
Recall from Section~\ref{secnot} that if the kernel curve is degenerate (resp.\ has genus $0$), the generating function is algebraic (resp.\ algebraic or differentially transcendental) in its three variables. So let us  assume that the kernel defines an elliptic curve. 

By \cite[Thm~1.1]{dreyfus2021differential}, the series is $\partial_x$-algebraic if and only if it is $\partial_y$-algebraic, if and only if it is $\partial_t$-algebraic.  Then, when  the series is $\partial_x$-transcendental, it is  $\partial_y$-transcendental and $\partial_t$-transcendental. Hence it is not D-finite in each of its variables. Similarly, when the series is $\partial_y$-transcendental (resp.\ $\partial_t$-transcendental), it is not D-finite in each of its variables. So Theorem \ref{thm4} holds when the series is differentially transcendental in one of its variables. 

So we can focus on the situation where the series is differentially algebraic in each of its variables.
By \eqref{eq:funcequ}, it suffices to show that $Q(x,0,t)$ and $Q(0,y,t)$ are not D-finite in each of their variables. Let us focus on $Q(x,0,t)$, the proof for $Q(0,y,t)$ being similar. 

Let us begin by the non-D-finiteness in $x$. The strategy is heavily inspired by the one used in \cite{KuRa12}. For any value of $0<t<1$, the function $Q(x,0,t)$ admits a meromorphic continuation, see \eqref{eq:liftings_GF}.  Using \cite[Prop.~3.9]{hardouin2021differentially}, the continuation of $K(x,0,t)Q(x,0,t)$ admits the form 
\begin{equation*}
    r_x (\omega,t)=f(x(\omega,t),t)+g(\omega,t),
\end{equation*}
where $f\in \C(x,t)$ is a rational function (called decoupling function)\ and $g$ is $\sigma$-invariant, meaning $\omega_3$-periodic. Since $r_x (\omega,t)$ and $f(x(\omega,t),t)$ are $\omega_1(t)$-periodic, the same holds for $g(\omega,t)$ proving that it is elliptic with the periods $\omega_1(t),\omega_3 (t)$.
Let $F(x,0,t)$ be the pullback of  $g(\omega,t)$. 
To the contrary, assume that we have a D-finite relation 
\begin{equation}
\label{eq:contradiction_identity_DF}
   \sum_{i=0}^n b_i (x,t)  \partial_x^i F(x,0,t)=0, \quad b_i\in \C[x,t], \quad b_n\neq 0.
\end{equation}
Fix $t_0\in(0,\varepsilon)$ such that the group is infinite and $(t-t_0 )$ does not divide $b_n (x,t)$. Recall that for all $0<t_0<\varepsilon$, $Q(x,0,t_0)\notin \C(x)$ so $\omega\mapsto g(\omega,t_{0})$ is not constant.  Then $g(\omega,t_0)$ admits at least one pole and since it is $\omega_3 (t_0)$-periodic and $\omega_2(t_0)/\omega_3(t_0)\notin \Q$ it admits an infinite number of poles modulo $\omega_2(t_0)\Z$. The pullback $F(x,0,t_0)$ is a multivalued function in $x$ and the set of the projection of the poles forms an infinite set. 
By the Cauchy-Lipschitz theorem, the poles of $F$ must be located at the zeros of $b_n (x,t_0)$. Since there are an infinite number of such poles, we find that $(t-t_0 )$ divides $b_n (x,t_0)$. A contradiction. 

Consider the $t$-holonomy.
To the contrary, assume that we have a D-finite relation 
\begin{equation}
\label{eq:tcontradiction_identity_DF}
   \sum_{i=0}^n b_i (x,t)  \partial_t^i F(x,0,t)=0, \quad b_i\in \C[x,t], \quad b_n\neq 0.
\end{equation}
Again we fix $t_{0}\in(0,\varepsilon)$ in the $t$-plane such that $(t-t_0)$ does not divide $b_n (x,t)$, and the group is infinite for this value of $t$ (that is $\omega_{3}(t_{0})/\omega_{2}(t_{0})\notin\mathbb{Q}$). We repeat the above reasoning to deduce that $F(x,0,t_{0})$ is a multivalued function in $x$ and the projection of the poles form an infinite set. Moreover, due to the meromorphicity of $g(\omega,t)$, the differential equation \eqref{eq:tcontradiction_identity_DF} holds on every sheet of $F(x,0,t)$.   

Since $(t-t_0)$ does not divide $b_n (x,t)$, we may choose one pole $x_{c}\in\mathbb{C}$ such that $b_{n} (x_{c},t_{0})\neq 0$, and we may analyse $F(x_{c},0,t)$ for $t$ in the vicinity of $t_{0}$. Again by the Cauchy-Lipschitz theorem, if $t\mapsto F(x_c,0,t)$ is meromorphic, then the poles of $F(x_{c},0,t)$ must be located at the zeros of $b_n (x_c,t)$. Since $b_{n} (x_{c},t_{0})\neq 0$, we deduce that $t\mapsto F(x_c,0,t)$ is not meromorphic and therefore  $F(x_{c},0,t)=\infty$ for all $t$. Letting $\omega=b(t)$ denote the corresponding pole of $g(\omega,t)$, this implies that $x(b(t),t)$ is constant, that is, it doesn't depend on $t$. Similarly $x(b(t)+k\omega_{3}(t),t)$ is constant for infinitely many integers $k$, which implies that $x(b(t)+k\omega_{3}(t)+j\omega_{2}(t),t)$ is also constant for any integer $j$. 

Now, if $\frac{\omega_{3}(t)}{\omega_{2}(t)}$ is not constant, we can choose $t_{1}$ near $t_{0}$ such that 
\begin{equation*}
    \frac{\omega_{3}(t_{1})}{\omega_{2}(t_{1})} = \frac{j}{k} \in\mathbb{Q}.
\end{equation*}
The function $x(b(t)+mk\omega_{3}(t)-mj\omega_{2}(t),t)$ is constant for infinitely many integers $m$, and these functions are all equal at $t=t_{1}$, hence these functions are all equal to the same constant. Consider three such values $m_{1},m_{2},m_{3}$ and consider $t_{2}\neq t_1$ so that $\frac{\omega_{3}(t_{2})}{\omega_{2}(t_{2})} \neq \frac{j}{k}$ but such that $t_2$ is sufficiently close to $t_1$ such that the distinct values $z_{1},z_{2},z_{3}$ defined by $z_{i}=b(t_2)+m_i k\omega_{3}(t_2)-m_i j\omega_{2}(t_2)$ all lie in the same fundamental domain of $x(\omega,t_{2})$. Then the results above imply that $x(z_{1},t_{2})=x(z_{2},t_{2})=x(z_{3},t_{2})=x(b(t_{1}),t_{1})$, but this contradicts the fact that $x(\omega,t_{2})$ takes every value exactly twice in each fundamental domain.

Finally, if $\frac{\omega_{3}(t)}{\omega_{2}(t)}$ is constant, then it is irrational as we have assumed that the group is infinite. Now for any $\alpha\in\mathbb{R}$, we can choose a sequence $(j_{1},k_{1}),(j_{2},k_{2}),\ldots$ of pairs of integers such that $k_{n}\omega_{3}(t)-j_{n}\omega_{2}(t)\to \alpha\omega_{2}(t)$ as $n\to\infty$. Then 
\begin{equation*}
   x(b(t)+k_n\omega_{3}(t)-j_n\omega_{2}(t),t)\to x(b(t)+\alpha\omega_{2}(t),t)    
\end{equation*}
for each fixed $t$. Hence, $x(b(t)+\alpha\omega_{2}(t),t)$ does not depend on $t$, so it only depends on $\alpha$. Hence, for $t_{1}$ sufficiently close to $t_{0}$, we have
\[x(b(t_{0})+\alpha\omega_{2}(t_{0}),t_{0})-x(b(t_{1})+\alpha\omega_{2}(t_{1}),t_{1})=0,\]
for $\alpha\in\mathbb{R}$. Since this is a meromorphic function of $\alpha$, it must be $0$ for any $\alpha\in\mathbb{C}$, that is
\[x(b(t_{0})+\alpha\omega_{2}(t_{0}),t_{0})=x(b(t_{1})+\alpha\omega_{2}(t_{1}),t_{1}).\]
However, this implies that the two sides of this expression, as functions of $\alpha$, have the same minimal ratio of periods $\omega_{1}(t_0)/\omega_{2}(t_0)=\omega_{1}(t_1)/\omega_{2}(t_1)$. Since this applies for all $t_1$ sufficiently close to $t_0$, the ratio $\omega_{1}(t)/\omega_{2}(t)$ must not depend on $t$. This is a contradiction, however, as $\omega_{1}(t)/\omega_{2}(t)\to 0$ as $t\to 0$.
Indeed, using \cite[Eq.~(6.5)]{BeBMRa-21}
one has 
\begin{equation*}
    \frac{\omega_{1}(t)}{\omega_{2}(t)} = i\frac{K(\sqrt{1-k})}{K(\sqrt{k})}
\end{equation*}
and $k\to 1$ as $t\to0$, see \cite[Eq.~(7.26)]{KuRa12}. On the other hand, since the kernel is an elliptic curve, $\omega_{1}(t)\neq 0$, contradicting the fact that $\omega_{1}(t)/\omega_{2}(t)$ must not depend on $t$.
\end{proof}

\subsection*{Acknowledgments}
We thank Mireille Bousquet-M\'elou for interesting discussions at an early stage of the project. We thank Charlotte Hardouin and Michael F.\ Singer for very useful discussions.  
  
    \bibliographystyle{abbrv}
\bibliography{biblio}

\end{document}